\documentclass[12pt, a4paper, reqno]{amsart}

\allowdisplaybreaks

\usepackage[english]{babel}

\usepackage[T1]{fontenc}
\usepackage{lmodern}

\usepackage{amsmath}
\usepackage{amssymb}
\usepackage{amsthm}

\usepackage{url}

\usepackage{xcolor}
\usepackage{comment}
\usepackage{tikz-cd}
\usepackage{caption}
\usepackage{multirow}
\usepackage{multicol}

\usepackage{longtable}

\usepackage{hyperref} 

\usepackage{enumitem}
\usepackage[normalem]{ulem}

\renewcommand{\epsilon}{\varepsilon}
\renewcommand{\theta}[0]{\vartheta}
\renewcommand{\phi}[0]{\varphi}

\DeclareMathOperator{\GL}{GL}
\DeclareMathOperator{\SL}{SL}

\DeclareMathOperator{\End}{End}
\DeclareMathOperator{\Aut}{Aut}

\DeclareMathOperator{\Hol}{Hol}

\DeclareMathOperator{\NHol}{NHol}

\DeclareMathOperator{\res}{res}

\newtheorem{dummy}{Dummy}
\numberwithin{dummy}{section}
\numberwithin{figure}{section}

\newtheorem{theorem}[dummy]{Theorem}

\newtheorem{lemma}[dummy]{Lemma}

\newtheorem{prop}[dummy]{Proposition}

\theoremstyle{definition}

\newtheorem*{notationstar}{Notation}

\newtheorem{assume}[dummy]{Assumption}

\theoremstyle{remark}

\newtheorem{remark}[dummy]{Remark}

\makeatletter
\def\imod#1{\allowbreak\mkern10mu({\operator@font mod}\,\,#1)}
\makeatother

\makeatletter
\newcommand{\extp}{\@ifnextchar^\@extp{\@extp^{\,}}}
\def\@extp^#1{\mathop{\bigwedge\nolimits^{\!#1}}}
\makeatother

\numberwithin{equation}{section}

\usepackage{amsthm, mathtools, amssymb, amsfonts, stmaryrd, latexsym,
mathrsfs, todonotes, color}

\begin{document}

\date{31 August 2023, 12:00 CET --- Version 1.08%
}

\title[$p$-groups of class two with a small multiple holomorph]
      {Finite $p$-groups of class two with\\ a small multiple holomorph}
      
\author{A.~Caranti}

 \address[A.~Caranti]%
  {Dipartimento di Matematica\\
   Universit\`a degli Studi di Trento\\
   via Sommarive 14\\
   I-38123 Trento\\
   Italy\\\endgraf
   ORCiD: 0000-0002-5746-9294} 

 \email{andrea.caranti@unitn.it} 
 \urladdr{https://caranti.maths.unitn.it/}

\author{Cindy (Sin Yi) Tsang}

 \address[C.~Tsang]%
  {Department of Mathematics\\
   Ochanomizu University\\
   2-1-1 Otsuka\\
   Bunkyo-ku, Tokyo\\
   Japan\\\endgraf
   ORCiD: 0000-0003-1240-8102} 

 \email{tsang.sin.yi@ocha.ac.jp} 
 \urladdr{https://sites.google.com/site/cindysinyitsang/}

\subjclass[2020]{20B35 20D15 20D45}

 \keywords{holomorph, multiple holomorph, regular subgroups, finite $p$-groups of class two, bilinear forms}

\begin{abstract}We consider the quotient group $T(G)$ of the multiple holomorph by the holomorph of a finite $p$-group $G$ of class two for an odd prime $p$. By work of the first-named author, we know that $T(G)$ contains a cyclic subgroup of order $p^{r-1}(p-1)$, where $p^r$ is the exponent of the quotient of $G$ by its center. In this paper, we shall exhibit examples of $G$ (with $r = 1$) such that $T(G)$ has order exactly $p-1$, which is as small as possible.
\end{abstract}

 \thanks{
  The first-named author is a member of GNSAGA--INdAM, Italy, and acknowledges support from the Department of
Mathematics of the University of Trento.  
\\Both authors acknowledge that
this research was supported by JSPS KAKENHI Grant Number 21K20319. We also thank the referee for helpful comments.}

\maketitle

\thispagestyle{empty}

\section{Introduction}

Let $G$ be any group, and write $S(G)$ for the group of permutations on $G$, where maps are composed from left to right.  Let
\[ \rho : G\rightarrow S(G);\,\ x^{\rho(y)} = xy\]
denote the right regular representation of $G$. The  \emph{holomorph} of $G$ may be defined as the subgroup
\[ \Hol(G) = \Aut(G)\rho(G) = \Aut(G)\ltimes \rho(G)\]
of $S(G)$, or equivalently, the normalizer of $\rho(G)$ in $S(G)$. The \emph{multiple holomorph} of $G$, which is denoted as $\mathrm{NHol}(G)$, is in turn defined to be the normalizer of $\Hol(G)$ in $S(G)$. It is easy to see that the quotient
\[ T(G) = \NHol(G)/\Hol(G)\]
acts regularly, via conjugation, on the regular subgroups of $S(G)$  that are isomorphic to $G$ and have normalizer equal to $\Hol(G)$. These subgroups are exactly the normal regular subgroups of $\Hol(G)$ isomorphic to $G$ when $G$ is finite, which is the case of interest of this paper.

The group $T(G)$ was first studied in \cite{Miller} and has attracted more attention recently since \cite{Kohl} was published. The structure of $T(G)$ has been computed for certain families of groups $G$, such as 
\begin{itemize}
\item finitely generated abelian groups (\cite{Mills},\cite{fg}),
\item dihedral and generalized quaternion groups (\cite{Kohl}),
\item finite groups of squarefree order (\cite{squarefree}).
\end{itemize}
In this paper, we shall consider finite $p$-groups $G$ of class two, where $p$ always denotes an odd prime. In this case, by \cite[Proposition 3.1]{class2}, we know that  
$T(G)$ has a cyclic subgroup of order $p^{r-1}(p-1)$, where $p^r$ is the exponent of the quotient of $G$ by its center. We shall say that the order of $T(G)$ is \emph{minimal} when it is exactly $p^{r-1}(p-1)$.

\begin{notationstar}For any group $G$, we employ the following notation:
\begin{itemize}
\item $G'=\mbox{the commutator subgroup of $G$}$.
\item $Z(G)=\mbox{the center of $G$}$.
\item $\mathrm{Frat}(G)=\mbox{the Frattini subgroup of $G$}$.
\item $\Aut_c(G)=\mbox{the kernel of the natural homomorphism}$
\[\Aut(G)\rightarrow \Aut(G/Z(G)).\]
Also write $\Aut^c(G)$ for the image of this homomorphism, which we shall sometimes identify as $\Aut(G)/\Aut_c(G)$.
\item $\Aut_z(G)=\mbox{the kernel of the natural homomorphism}$
\[  \Aut(G)\rightarrow \Aut(Z(G)).\]
\item In the case that $G'=Z(G)$, we have $\Aut_c(G)\subseteq \Aut_z(G)$, and this allows us to define a natural homomorphism
\[ \Aut(G/G')\rightarrow \Aut(G');\,\ \alpha\mapsto \hat{\alpha}.\]
Here $\hat{\alpha}$ is the element induced by any lift of $\alpha$ in $\Aut(G)$.
\end{itemize}
For any $x,y\in G$, we shall also write $x^y = y^{-1}xy$ and $[x,y] = x^{-1}x^y$.
\end{notationstar}

Let us now return to finite $p$-groups $G$ of class two. In this case, by \cite[Section 2]{class2}, we may study certain elements of $T(G)$ via the use of bilinear forms. In our previous work \cite[Section 4]{LMH}, we exploited this approach and showed that
\begin{equation}\label{T(G) pre} T(G)= \mathcal{S}\rtimes \mathcal{S}'\mbox{ when $G' = Z(G)$ and $\Aut^c(G)=1$}.\end{equation}
Here, loosely speaking, the subgroups $\mathcal{S}$ and $\mathcal{S}'$ consist of the elements corresponding to symmetric and anti-symmetric forms, respectively. In Section \ref{bilinear form sec}, we shall generalize this fact and show that (\ref{T(G) pre}) is still valid even when $\Aut^c(G)\neq 1$, provided that there is no non-trivial $\Aut^c(G)$-equivariant homomorphism from $G/G'$ to $\Aut^c(G)$ (Assumption \ref{assumption}).

In \cite{LMH}, we considered a family of $p$-groups $G$ of class two which may be constructed from linear maps $\pi : V\rightarrow\Lambda^2V$. Here $V$ is a finite dimensional vector space over $\mathbb{F}_p$ and $\Lambda^2V$ denotes its exterior square. The construction is from \cite{CarIJM}, to be reviewed in Section \ref{group section}. Let us note that the constructed $p$-groups $G$ have order 
\[ |G| = p^{n + {n\choose 2}}\mbox{ when }n = \dim_{\mathbb{F}_p}(V),\]
and are \emph{special} in the sense that
\[G' = Z(G) = \mathrm{Frat}(G).\]
By \cite{simple-constr}, the condition $\Aut^c(G)=1$ may be realized by certain $\pi$ of full rank for any $n\geq 4$. Using these special $p$-groups $G$ and (\ref{T(G) pre}), we proved in \cite{LMH}  that $T(G)$ can be made arbitrarily large.

In the present paper, we consider the opposite direction and search for examples of $G$ such that $T(G)$ is small (or even of minimal order). To get such examples, we naturally want $\Aut^c(G)$ to be large (see the discussion after Remark \ref{remark}), and this is more likely to happen when $\pi$ has small rank (see the discussion after (\ref{iden Aut})). 

The extreme case is when $\pi$ is the trivial linear map (when the rank is zero). The associated group $G$ as defined in (\ref{Gpi}) is the free $p$-group of class two and exponent $p$. We have that
\[ \Aut^c(G) = \Aut(G/G') \]
is as large as possible, and it is already known by \cite[Theorem 5.2]{class2} that $T(G)$ has minimal order $p-1$.

The next natural case to consider is when the rank of $\pi$ is one. For $n=2$, the associated $G$ has order $p^3$, and we already know by \cite[Proposition 5.1]{class2} that  $T(G)$ has minimal order $p-1$. For $n=3,4$, as we shall show in Propositions \ref{rank one prop'} and \ref{rank one prop}, respectively, up to a choice of basis there are only two and three possibilities of $\pi$ of rank one.  It turns out that $T(G)$ has minimal order $p-1$ in four of the cases, and $T(G)$ has order $(p-1)^2$ in the remaining case. 
Our main result is the following:

\begin{theorem}\label{thm1} The following holds.
\begin{enumerate}[label = $(\alph*)$]
\item Let $p\geq 5$ be any prime. For the group
\begin{align*}
G =\Bigg\langle x_1,x_2,x_3: &\, [[x_i,x_j],x_k] = 1\mbox{ for }1\leq i,j,k\leq 3,\\[-10pt]
&x_2^p = x_3^p= 1,\, x_1^p = [x_1,x_2]\Bigg\rangle
\end{align*}
of order $p^6$, we have $T(G)\simeq \mathbb{F}_p^\times$.
\item Let $p\geq 3$ be any prime. For the group
\begin{align*}
G =\Bigg\langle x_1,x_2,x_3: &\, [[x_i,x_j],x_k] = 1\mbox{ for }1\leq i,j,k\leq 3,\\[-10pt]
&x_2^p = x_3^p= 1,\, x_1^p = [x_2,x_3]\Bigg\rangle
\end{align*}
of order $p^6$, we have $T(G)\simeq \mathbb{F}_p^\times$.
\item Let $p\geq 5$ be any prime. For the group
\begin{align*}
G =\Bigg\langle x_1,x_2,x_3,x_4 : &\, [[x_i,x_j],x_k] = 1\mbox{ for }1\leq i,j,k\leq 4,\\[-10pt]
&x_2^p = x_3^p=x_4^p = 1,\, x_1^p = [x_1,x_2]\Bigg\rangle
\end{align*}
of order $p^{10}$, we have $T(G)\simeq \mathbb{F}_p^\times$.
\item Let $p\geq 3$ be any prime. For the group
\begin{align*}
G =\Bigg\langle x_1,x_2,x_3,x_4 : &\, [[x_i,x_j],x_k] = 1\mbox{ for }1\leq i,j,k\leq 4,\\[-10pt]
&x_2^p = x_3^p=x_4^p = 1,\, x_1^p = [x_3,x_4]\Bigg\rangle
\end{align*}
of order $p^{10}$, we have $T(G)\simeq \mathbb{F}_p^\times$.
\item Let $p\geq 5$ be any prime. For the group
\begin{align*}
G =\Bigg\langle x_1,x_2,x_3,x_4 : &\, [[x_i,x_j],x_k] = 1\mbox{ for }1\leq i,j,k\leq 4,\\[-10pt]
&x_2^p = x_3^p=x_4^p = 1,\, x_1^p = [x_1,x_2][x_3,x_4]\Bigg\rangle
\end{align*}
of order $p^{10}$, we have $T(G)\simeq\mathbb{F}_p^\times \times \mathbb{F}_p^\times$.
\end{enumerate}
\end{theorem}


\section{Multiple holomorph via bilinear forms}\label{bilinear form sec}

Let $p$ be an odd prime and let $G$ be a finite $p$-group of class two for which $G'=Z(G)$. Notice that then $\Aut_c(G)\subseteq \Aut_z(G)$. In this case, by \cite{LMH}, the structure of $T(G)$ may be studied using certain bilinear forms, as follows.

First, clearly any regular subgroup of $\Hol(G)$ must take the form
\[ N_\gamma = \{\gamma(x)\rho(x): x\in G\},\mbox{ where }\gamma :G \rightarrow\Aut(G)\]
is some suitable map. As shown in \cite[Theorem 5.2]{perfect}, the set $N_\gamma$ is a normal regular subgroup of $\Hol(G)$ exactly when
\[ \begin{cases}
 \gamma(xy) =\gamma(y)\gamma(x)\\
 \gamma(x^\beta) = \gamma(x)^\beta
 \end{cases}
 \mbox{for all }x,y\in G\mbox{ and } \beta\in \Aut(G),\]
namely when $\gamma$ is an $\Aut(G)$-equivariant anti-homomorphism. In our setting of $G$, it follows from \cite[Proposition 2.2]{class2} that such $\gamma$ with image lying inside $\Aut_c(G)$ may be parametrized by $G'$-valued bilinear forms on $G/G'$ which are also $\Aut(G)$-equivariant.

\begin{prop}
The following data are equivalent.
\begin{enumerate}[label = $(\arabic*)$]
\item A normal regular subgroup $N$ of $\Hol(G)$ such that its projection onto $\Aut(G)$ along $\rho(G)$ is contained in $\Aut_c(G)$.
\item An anti-homomorphism $\gamma : G\rightarrow\Aut_c(G)$ such that
\[\gamma(x^\beta) = \gamma(x)^\beta\]
for all $x\in G$ and $\beta\in \Aut(G)$.
\item A bilinear form $\Delta : G/G'\times G/G'\rightarrow G'$ such that
\begin{equation}\label{Delta1} \Delta(u^\beta,v^\beta) = \Delta(u,v)^\beta \end{equation}
for all $u,v\in G/G'$ and $\beta\in \Aut(G)$.
\end{enumerate}
The data $(1),(2)$ are related via $N = N_\gamma$, and $(2),(3)$ are related via  
\begin{equation}\label{gamma-Delta}x^{\gamma(y)} = x\Delta(xG',yG')\end{equation}
for all $x,y\in G$.
\end{prop}

Not every $\Aut(G)$-equivariant anti-homomorphism  $\gamma : G\rightarrow\Aut(G)$ would have image lying inside $\Aut_c(G)$ in general. But for $x\in G'$, we know from \cite[Lemma 1.7]{class2} that $\gamma(x)$ is always an inner automorphism. Since $G$ is of class two, all inner automorphisms are central. It follows that $\gamma$ induces a well-defined map
\[ \overline{\gamma}: G/G'\rightarrow \Aut^c(G);\,\ \overline{\gamma}(xG') = \gamma(x)\Aut_c(G).\]
This map is clearly an anti-homomorphism, or simply homomorphism because the domain $G/G'$ is abelian, satisfying
\[ \overline{\gamma}(v^\alpha) = \overline{\gamma}(v)^\alpha\]
for all $v\in G/G'$ and $\alpha\in\Aut^c(G)$, namely it is an $\Aut^c(G)$-equivariant homomorphism. We note that the image of $\gamma$ lies inside $\Aut_c(G)$ if and only if the induced map $\overline{\gamma}$ is trivial.

In our setting of $G$, we have the inclusion $\Aut_c(G)\subseteq \Aut_z(G)$, and so the condition (\ref{Delta1}) is vacuous when $\beta\in \Aut_c(G)$. This means that we only need to consider the action of $\Aut^c(G)$ here.

For simplicity, henceforth, let us further assume the following:

\begin{assume}\label{assumption}There is no non-trivial $\Aut^c(G)$-equivariant homomorphism from $G/G'$ to $\Aut^c(G)$.\end{assume}

From the above discussion, we then deduce that:

\begin{prop}\label{gamma prop new}
Under Assumption \ref{assumption}, the following data are equivalent.
\begin{enumerate}[label = $(\arabic*)$]
\item A normal regular subgroup $N$ of $\Hol(G)$.
\item A bilinear form $\Delta : G/G'\times G/G'\rightarrow G'$ such that
\begin{equation}\label{Delta2} \Delta(u^\alpha,v^\alpha) = \Delta(u,v)^{\hat{\alpha}} \end{equation}
for all $u,v\in G/G'$ and $\alpha\in \Aut^c(G)$. 
\end{enumerate}
\end{prop}

Before proceeding, let us set up some notation. Define
\[ B =  \{\mbox{bilinear forms $\Delta :G/G'\times G/G'\rightarrow G'$ satisfying (\ref{Delta2})}\},\]
which is an abelian group under pointwise multiplication in $G'$. Let $S$ and $S'$, respectively, be the subgroups of $B$ consisting of the bilinear forms which are symmetric and anti-symmetric, namely
\begin{align*}
S & = \{\Delta\in B \mid \Delta(u,v) = \Delta(v,u) \mbox{ for all }u,v\in G/G'\},\\
S' & = \{\Delta\in B\mid \Delta(u,v) = \Delta(v,u)^{-1}\mbox{ for all }u,v\in G/G'\}.
\end{align*}Every $\Delta\in B$ may be decomposed as
\[ \Delta(u,v) = (\Delta(u,v)\Delta(v,u))^{1/2} \cdot (\Delta(u,v)\Delta(v,u)^{-1})^{1/2}.\]
We then deduce that $B = SS'$ with $S\cap S'=\{\Delta_0\}$ is an inner direct product of $S$ and $S'$, where $\Delta_0$ denotes the trivial bilinear form which sends everything to the identity element.

Each bilinear form $\Delta\in B$ gives rise to an $\Aut(G)$-equivariant anti-homomorphism $\gamma : G\rightarrow \Aut(G)$ via (\ref{gamma-Delta}). Let us write
 $N_\Delta = N_\gamma$ for the associated normal regular subgroup of $\Hol(G)$, which corresponds to an element of $T(G)$ if and only if $N_\Delta$ is isomorphic to $G$. We know by \cite[Proposition 3.1]{LMH} that the isomorphism class of $N_\Delta$ depends only on the anti-symmetric component of $\Delta$. Since $N_{\Delta_0}$ is simply $\rho(G)$, as shown in \cite[Corollary 3.2]{LMH}, we have $N_\Delta\simeq G$ for all symmetric forms $\Delta \in S$. For anti-symmetric forms $\Delta\in S'$, things are much more complicated and it seems to us that not much can be said in general, except the few facts proven in \cite[Proposition 3.3]{LMH} (also see Remark \ref{remark}). But for a special family (\ref{Gpi}) of groups $G$ that we shall consider in Section \ref{group section}, there is a linear algebra method to determine whether $N_\Delta$ is isomorphic to $G$ for anti-symmetric forms $\Delta \in S'$ (see Proposition \ref{criterion} for the criterion).
 
Given any $\theta\Hol(G) \in T(G)$, consider its corresponding normal regular subgroup $\rho(G)^\theta$ of $\Hol(G)$. By Proposition \ref{gamma prop new}, there is a unique bilinear form $\Delta_\theta\in B$ for which $\rho(G)^\theta = N_{\Delta_\theta}$. Let us put
\begin{align*}
\mathcal{S} & = \{\theta\Hol(G)\in T(G)\mid\Delta_\theta\in S\},\\
\mathcal{S}' & = \{\theta\Hol(G) \in T(G) \mid \Delta_\theta\in S'\}.
\end{align*}
It was shown in \cite[Section 4]{LMH} that
\[ T(G)= \mathcal{S}\rtimes \mathcal{S}'\mbox{ whenever }\Aut^c(G) =1.\]
But the condition $\Aut^c(G) =1$ was imposed there only for the sake of simplicity, and essentially the same argument shows that this in fact holds more generally. Some modifications do need to be made, so let us explain this carefully.

Given any $\theta\Hol(G)\in T(G)$ with $1^\theta = 1$, the exact same argument as in \cite{LMH} shows that $\theta$ induces, via restriction, automorphisms
\begin{equation}\label{res def} \res_c(\theta) : G/G' \rightarrow G/G'\mbox{ and }\res_z(\theta): G'\rightarrow G'. \end{equation}
Notice that the coset representative $\theta$ is unique up to $\Aut(G)$. In the case that $\Aut^c(G)=1$, both of the above are then independent of the choice of $\theta$, so as in \cite{LMH} one gets a well-defined homomorphism
\begin{align}\label{old res}
\res: T(G) &\rightarrow \Aut(G/G')\times \Aut(G')\\\notag
 \theta\Hol(G) &\mapsto (\res_c(\theta),\res_z(\theta)).
 \end{align}
 However, this is no longer well-defined when $\Aut^c(G)\neq 1$ and we need to change the range of $\res$ to avoid this problem. 

Let $\Gamma(G)$ denote the image of the  
 natural homomorphism
\begin{align*}
\Aut(G)&\rightarrow \Aut(G/G')\times \Aut(G')\\
\beta&\mapsto (\res_c(\beta),\res_z(\beta))
\end{align*}
induced by restrictions. Equivalently, in the current setting
\[ \Gamma(G) = \{(\alpha,\hat{\alpha}) : \alpha \in \Aut^c(G)\}.\]
Let $\mathrm{Norm}(\Gamma(G))$ denote its normalizer in $\Aut(G/G')\times \Aut(G')$.
 
\begin{lemma}\label{res image}For any $\theta\Hol(G)\in T(G)$ with $1^\theta = 1$, we have
\[ (\res_c(\theta),\res_z(\theta))\in \mathrm{Norm}(\Gamma(G)).\]
\end{lemma}

\begin{proof}Note that $\theta$ is an element of $\NHol(G)$. This means that $\theta$ normalizes $\Hol(G)$, but it fixes the identity, so it also normalizes $\Aut(G)$. For any $\beta\in \Aut(G)$, we then have $\theta^{-1}\beta\theta\in \Aut(G)$ and so
\begin{align*}
&(\res_c(\theta),\res_z(\theta))^{-1}(\res_c(\beta),\res_z(\beta))(\res_c(\theta),\res_z(\theta)) \\
&\hspace{5cm}=
(\res_c(\theta^{-1}\beta\theta),\res_z(\theta^{-1}\beta\theta))
\end{align*}
is again an element of $\Gamma(G)$.
\end{proof}

Returning to the discussion in (\ref{res def}), since the $\theta$ there is unique up to an element of $\Aut(G)$, the coset
\[ (\res_c(\theta) ,\res_z(\theta))\Gamma(G)\]
does not depend on the choice of $\theta$. Thus, by Lemma \ref{res image}, we obtain a well-defined homomorphism
\begin{align}\label{new res}
 \res : T(G) &\rightarrow \mathrm{Norm}(\Gamma(G))/\Gamma(G)\\\notag
 \theta\Hol(G)&\mapsto (\res_c(\theta),\res_z(\theta)) \Gamma(G).
 \end{align}
 In the case that $\Aut^c(G)=1$, the subgroup $\Gamma(G)$ is trivial and we get back the map (\ref{old res}) that we had in \cite{LMH}.

\begin{lemma}\label{action lemma}
For any $\Delta\in B$ and $(\alpha_c,\alpha_z)\in \mathrm{Norm}(\Gamma(G))$, define
\[ \Delta^{(\alpha_c,\alpha_z)}(u,v) = \Delta(u^{\alpha_c^{-1}},v^{\alpha_c^{-1}})^{\alpha_z}\]
for all $u,v\in G/G'$. Then $\Delta^{(\alpha_c,\alpha_z)}$ also belongs to $B$.
\end{lemma}

\begin{proof}Clearly $\Delta^{(\alpha_c,\alpha_z)}$ is a $G'$-valued bilinear form on $G/G'$, and the issue here is whether $\Delta^{(\alpha_c,\alpha_z)}$ also satisfies (\ref{Delta2}), or equivalently (\ref{Delta1}). Let $\beta\in \Aut(G)$, and since $(\alpha_c,\alpha_z)\in \mathrm{Norm}(\Gamma(G))$, we have
\[ (\alpha_c\res_c(\beta)\alpha_c^{-1},\alpha_z\res_z(\beta)\alpha_z^{-1}) = (\res_c(\beta_0),\res_z(\beta_0))
\]
for some $\beta_0\in \Aut(G)$. For any $u,v\in G/G'$, we then have
\begin{align*}
\Delta^{(\alpha_c,\alpha_z)}(u^\beta,v^\beta) 
& = \Delta(u^{\beta\alpha_c^{-1}},v^{\beta\alpha_c^{-1}})^{\alpha_z}\\
& = \Delta(u^{\alpha_c^{-1}\cdot \alpha_c\beta\alpha_c^{-1}}, v^{\alpha_c^{-1}\cdot \alpha_c\beta\alpha_c^{-1}})^{\alpha_z}\\
& = \Delta(u^{\alpha_c^{-1}\beta_0},v^{\alpha_c^{-1}\beta_0})^{\alpha_z}\\
& = \Delta(u^{\alpha_c^{-1}},v^{\alpha_c^{-1}})^{\beta_0\alpha_z} \\
&= \Delta(u^{\alpha_c^{-1}},v^{\alpha_c^{-1}})^{\alpha_z\beta}\\
& = \Delta^{(\alpha_c,\alpha_z)}(u,v)^\beta,
\end{align*}
where in the third last equality, we used the fact that $\Delta\in B$ satisfies (\ref{Delta1}). This shows that indeed $\Delta^{(\alpha_c,\alpha_z)}$ is an element of $B$.
\end{proof}

Lemma \ref{action lemma} shows that we have a right action of $\mathrm{Norm}(\Gamma(G))$ on the abelian group $B$. Notice that both $S$ and $S'$ are invariant under this action. Moreover, since every $\Delta\in B$ satisfies (\ref{Delta2}), this action factors through $\Gamma(G)$. Multiplication of elements in $T(G)$, when translated to the corresponding bilinear forms, may then be expressed as follows.

\begin{prop}\label{multiplication prop} For any $\theta_1\Hol(G),\theta_2\Hol(G)\in T(G)$, we have
\[ \Delta_{\theta_1\theta_2} = \Delta_{\theta_1}^{\res(\theta_2)}\Delta_{\theta_2}.\]
\end{prop}

\begin{proof}
The proof of \cite[Proposition 4.1]{LMH} verbatim.
\end{proof}

\begin{prop}\label{kernel prop}We have $\mathcal{S} = \ker(\res)$.
\end{prop}
\begin{proof}The proof of \cite[Proposition 4.2]{LMH} verbatim, except that for our generalized $\res$, when $\theta\Hol(G)\in \ker(\res)$ we only know that there exists $\beta\in \Aut(G)$ such that
\[ (\res_c(\theta),\res_z(\theta) ) = (\res_c(\beta),\res_z(\beta)).\]
But we simply replace the equation in \cite[Proposition 4.2]{LMH} by
\[ [x^\beta,y^\beta]=[x,y]^\beta = [x,y]^\theta = [x^\theta,y^\theta]_\circ = [x^\beta,y^\beta]_\circ\mbox{ for all }x,y\in G\]
and the same argument shows that $\Delta_{\theta}$ is symmetric.
 \end{proof}

\begin{prop}\label{sd prop} We have $T(G) = \mathcal{S}\rtimes \mathcal{S}'$.
\end{prop}

\begin{proof}
We have $\mathcal{S}\cap\mathcal{S}' =1$ because $S\cap S' = 1$ and the trivial bilinear form corresponds to the identity element of $T(G)$.
Since $\mathcal{S}$ is a normal subgroup of $T(G)$ by Proposition \ref{kernel prop}, it remains to verify that every element $\theta\Hol(G)\in T(G)$ belongs to the product $\mathcal{S}\mathcal{S}'$.

 To do so, we first decompose $\Delta_\theta = \Delta\Delta_2$ with $\Delta\in S$ and $\Delta_2\in S'$. We have $N_{\Delta_\theta}\simeq G$ because $\theta\Hol(G)\in T(G)$. Then $N_{\Delta_2}\simeq G$ holds by \cite[Proposition 3.1]{LMH}, which means that 
\[ \Delta_ 2 = \Delta_{\theta_2}\mbox{ for some }\theta_2\Hol(G)\in T(G).\]
Consider $\Delta_1 = \Delta^{\res(\theta_2)^{-1}}$, which is an element of $S$. By \cite[Corollary 3.2]{LMH}, we have  $N_{\Delta_1}\simeq G$ and so
\[ \Delta_1 = \Delta_{\theta_1}\mbox{ for some }\theta_1\Hol(G)\in T(G).\]
From Proposition \ref{multiplication prop}, we then deduce that
\[ \Delta_{\theta_1\theta_2} = \Delta_{\theta_1}^{\res(\theta_2)}\Delta_{\theta_2} =\Delta\Delta_2 =\Delta_\theta.\]
This shows that $\theta\Hol(G) = \theta_1\Hol(G)\cdot \theta_2\Hol(G)$, which is an element of $\mathcal{S}\mathcal{S}'$ because $\Delta_{\theta_1}\in S$ and $\Delta_{\theta_2}\in S'$.
 \end{proof}

Finally, we also have:

\begin{theorem}\label{pre thm}We have $T(G)\simeq S\rtimes \res(\mathcal{S}')$.
\end{theorem}

\begin{proof}The proof of \cite[Theorem 4.4]{LMH} verbatim.
\end{proof}


Regarding the structure of $T(G)$, for the symmetric part, we simply have to compute $S$ and this gives us a normal subgroup of $T(G)$. For the anti-symmetric part, however, the regular subgroup arising from a $\Delta\in S'$ need not be isomorphic to $G$. The multiplication in $S'$ does not correspond to that in $T(G)$ either, so in addition to computing $S'$, we also need to understand the structure of $\res(\mathcal{S}')$.

\begin{remark}\label{remark} As observed in \cite[Example 2.6]{LMH}, each $\sigma\in \End(G')$ yields an anti-symmetric bilinear form
\[ \Delta_{[\sigma]}: G/G' \times G/G' \rightarrow G';\,\ \Delta_{[\sigma]}(xG',yG') = [x,y]^\sigma.\]
In our setting, this is an element of $S'$ if and only if 
\begin{equation}\label{sigma commute} \sigma \hat{\alpha} = \hat{\alpha}\sigma\mbox{ for all }\alpha\in \Aut^c(G).\end{equation}
For each $\lambda\in \mathbb{Z}$, we in particular have
\[ \Delta_{[\lambda]} : G/G'\times G/G'\rightarrow G';\;\ \Delta_{[\lambda]}(xG',yG') = [x,y]^\lambda,\]
which is clearly an element of $S'$. These bilinear forms were first considered in \cite[Section 3]{class2}, where it was shown that
\[  N_{\Delta_{[\lambda]}} \simeq G\iff \lambda\not\equiv -\tfrac{1}{2}\hspace{-3mm}\pmod{p}.\]
In this case, the corresponding element of $T(G)$ is the power map
\[ \theta_{\kappa} : G\rightarrow G;\,\ x^{\theta_{\kappa}} = x^{\kappa},\]
where $\kappa$ is the multiplicative inverse of $1+2\lambda$ modulo the exponent $p^r$ of $G/Z(G)$. These power maps give rise to the cyclic subgroup
\begin{equation}\label{theta}
\{ \theta_\kappa : \kappa\in \mathbb{Z}\mbox{ coprime to }p\} \simeq (\mathbb{Z}/p^r\mathbb{Z})^\times\end{equation}
of $T(G)$, and this is exactly the content of \cite[Proposition 3.1]{class2}. For an arbitrary $\sigma\in \End(G')$, however, it appears that there is no simple way to determine whether $N_{\Delta_{[\sigma]}} \simeq G$ in general, except the necessary condition $1+2\sigma\in \Aut(G')$ as shown in \cite[Example 3.4]{LMH}. But as we previously noted, for a special family (\ref{Gpi}) of groups $G$, we have a criterion (to be stated in Proposition \ref{criterion}) which can be checked using linear algebra techniques.\end{remark}

In the case that $\Aut^c(G) =1$, the condition (\ref{Delta2}) is vacuous, and $B$ is the set of all $G'$-valued bilinear forms on $G/G'$. Then $T(G)$ is likely to be large, and this is how we found examples of $G$ with large $T(G)$ in \cite{LMH}. To get examples of $G$ with a small $T(G)$, we shall consider the case when $\Aut^c(G)$ is large so that (\ref{Delta2}) yields many relations. In fact, for the groups $G$ in Theorem \ref{thm1}, there are enough relations from (\ref{Delta2}) that $S=1$ is forced to be trivial, while $\mathcal{S}'\simeq \res(\mathcal{S}')$ is forced to be the subgroup (\ref{theta})  in (a) -- (d) and is only slightly bigger in (e).

\section{A special family of finite $p$-groups of class two}\label{group section}

Let $p$ be an odd prime and let $n\geq 2$ be an integer. Let us review a construction, via linear algebra as described in \cite{CarIJM}, of $p$-groups of order $p^{n+{n\choose 2}}$ and class two.

Let $V$ be an $n$-dimensional vector space over $\mathbb{F}_p$ and let $\Lambda^2 V$ denote its exterior square. We shall write their additions multiplicatively and their scalar multiplications exponentially.

Now, let $\pi : V\rightarrow \Lambda^2V$ be any linear map. By fixing an $\mathbb{F}_p$-basis 
\[v_1,v_2,\dots,v_n\]
for $V$, in which case we have that
\[v_j \wedge v_k,\,\ 1\leq j < k \leq n\]
is an $\mathbb{F}_p$-basis for $\Lambda^2 V$, we may write $\pi$ in matrix form
\[ (\pi_{i,(j,k)}),\mbox{ where }1\leq i\leq n\mbox{ and }1\leq j<k\leq n,\]
with entries in $\mathbb{F}_p$. More precisely, for each $1\leq i\leq n$, we have
\[v_i^\pi =  \prod_{j<k} (v_j\wedge v_k)^{\pi_{i,(j,k)}}.\]
We can then construct a $p$-group $G = G_\pi$ via the presentation
\begin{align}\label{Gpi}\notag
G =\Bigg\langle x_1,x_2,\dots,x_n : &\, [[x_i,x_j],x_k] = 1\mbox{ for }1\leq i,j,k\leq n,\\[-10pt]
&\,  x_i^p = \prod_{j<k} [x_j,x_k]^{\pi_{i,(j,k)}} \mbox{ for }1\leq i\leq n\Bigg\rangle.
\end{align}
This group has order $p^{n+{n\choose 2}}$ and is special, namely
\[ G' = Z(G) = \mathrm{Frat}(G).\]
Moreover, we may identify
\begin{equation}\label{iden}
 G/G' = V\mbox{ and } G' = \Lambda^2 V
 \end{equation}
by associating each $x_iG'$ with $v_i$ and each $[x_j,x_k]$ with $v_j\wedge v_k$. Then the linear map $\pi$ corresponds to the $p$th power map $G/G'\rightarrow G'$ of $G$. For each $\alpha\in \GL(V)$, let us write $ \hat{\alpha}\in \GL(\Lambda^2 V)$ for the natural linear map induced by $\alpha$, given by
\[(v_j\wedge v_k)^{\hat{\alpha}} = v_j^\alpha \wedge v_k^\alpha\mbox{ for all }1\leq j<k\leq n.\]
As shown in \cite[Section 3]{CarIJM}, by setting
\[ \Aut^c(\pi) = \{ \alpha\in \GL(V) \mid \pi\hat{\alpha} = \alpha\pi\},\]
we have a natural isomorphism
\begin{equation}\label{iden Aut}
\Aut^c(G) \simeq\Aut^c(\pi).
\end{equation}
In particular, for any $\beta\in \Aut(G)$ such that $\beta\Aut_c(G)$ corresponds to $\alpha\in \Aut^c(\pi)$, its actions on $G/G'$ and $G'$, respectively, are those given by $\alpha$ and $\hat{\alpha}$ via the identifications (\ref{iden}). Notice that $\pi\hat{\alpha} = \alpha\pi$ is likely to yield more relations on $\alpha$ when $\pi$ has high rank. To find examples with large $\Aut^c(\pi)$, naturally we want to consider $\pi$ of small rank, and this is why we restrict to $\pi$ of rank one in this paper.

The structure of $T(G)$ for such groups $G$ was considered in \cite{LMH}. As noted in \cite[Proposition 5.1]{LMH}, the universal property of the exterior square and (\ref{iden}) imply that the $G'$-valued anti-symmetric bilinear forms on $G/G'$ are precisely the $\Delta_{[\sigma]}$ in Remark \ref{remark}.

\begin{prop}\label{criterion}Let $\sigma\in \End(G')$ be such that $1+2\sigma\in \Aut(G')$ and that (\ref{sigma commute}) holds. For any $\eta\in \Aut(G/G')$, the following are equivalent:
\begin{enumerate}[label=$(\alph*)$]
\item There exists $\theta\Hol(G)\in T(G)$ with $1^\theta =1$ such that $\Delta_\theta = \Delta_{[\sigma]}$ and $\res_c(\theta) = \eta$.
\item The equality $\pi\hat{\eta}(1+2\sigma) = \eta\pi$ is satisfied.
\end{enumerate}
Moreover, in this case, we have
\[ \res(\theta \Hol(G)) = (\eta,\hat{\eta}(1+2\sigma))\Gamma(G).\]
\end{prop}

\begin{proof}
Basically the same proof as \cite[Proposition 5.2]{LMH}. The same argument is valid as long as we impose the condition (\ref{sigma commute}) and replace the map $\res$ there with our generalization (\ref{new res}).

Let us note that $(\eta,\hat{\eta}(1+2\sigma))$ is indeed an element of $\mathrm{Norm}(\Gamma(G))$. This is because  by the condition (\ref{sigma commute}), we have
\begin{align*}
& \hspace{5mm}(\eta,\hat{\eta}(1+2\sigma))(\alpha,\hat{\alpha})(\eta,\hat{\eta}(1+2\sigma))^{-1} \\
& = (\eta\alpha\eta^{-1}, \hat{\eta}(1+2\sigma) \hat{\alpha}(1+2\sigma)^{-1}\hat{\eta}^{-1})\\
&=(\eta\alpha\eta^{-1},\hat{\eta}\hat{\alpha}\hat{\eta}^{-1})
\end{align*}
for any $(\alpha,\hat{\alpha}) \in \Gamma(G)$ with $\alpha\in \Aut^c(G)$.

Let us also note that whether the equality in (b) holds depends only on the coset $\eta\Aut^c(G)$. Indeed, if $\pi\hat{\eta}(1+2\sigma) = \eta\pi$ is satisfied, then by the condition (\ref{sigma commute}) and the identification (\ref{iden Aut}), we have
\begin{align*}
\pi\widehat{\eta\alpha}(1+2\sigma) & = \pi \hat{\eta}\hat{\alpha}(1+2\sigma)\\
& = \pi\hat{\eta} (1+2\sigma) \hat{\alpha}\\
& = \eta\pi \hat{\alpha}\\
& = \eta\alpha\pi
\end{align*}
for any $\alpha\in \Aut^c(G)$. This corresponds to the fact that 
\[ \theta\Hol(G) = \theta\beta\Hol(G)\]
for any $\beta\in \Aut(G)$, and so the existence of $\theta\Hol(G)$ in (a) also only depends on the coset $\eta\Aut^c(G)$.
\end{proof}

Therefore, under Assumption \ref{assumption}, we deduce that
\begin{align*}
\res(\mathcal{S}') & = \{(\eta,\hat{\eta}(1+2\sigma))\Gamma(G) : \eta\in \Aut(G/G'),\, \sigma\in \End(G')\\
&\hspace{3.25cm}\mbox{satisfying $1+2\sigma\in \Aut(G')$ and (\ref{sigma commute})},\\
&\hspace{3.25cm}\mbox{the equation $\pi\hat{\eta}(1+2\sigma) =\eta\pi$ holds}\},
\end{align*}
or by making the change of variables $\tau = 1+2\sigma$, that
\begin{align}\notag
\res(\mathcal{S}') & = \{(\eta,\hat{\eta}\tau)\Gamma(G) : \eta\in \Aut(G/G'),\, \tau\in \Aut(G')\\\label{S'}
&\hspace{3.25cm}\mbox{satisfying (\ref{sigma commute}) and $\pi\hat{\eta}\tau =\eta\pi$}\},
\end{align}
where $\mathcal{S}'$ is the subgroup of $T(G)$ defined as in Section \ref{bilinear form sec}. This implies that the structure of $\res(\mathcal{S}')$ may be computed by solving $\pi\hat{\eta}\tau =\eta\pi$ for $\eta\in \Aut(G/G')$ for each of the $\tau\in\Aut(G')$ satisfying (\ref{sigma commute}), and this is essentially a problem in linear algebra.

In the next subsections, we shall restrict to $n=3,4$ and determine the linear maps $\pi: V\rightarrow \Lambda^2V$ of rank one, up to a change of basis. We shall use the following remark in the proof.

\begin{remark}It is a well-known fact that every element of $\Lambda^2V$ may be written as a sum (or product in our notation) of at most $\lfloor \frac{n}{2}\rfloor$ decomposable $2$-vectors (wedge products $u\wedge v$), and there are elements that attain this bound. This follows immediately from the natural identification of elements of $\Lambda^2V$ as $n\times n$ anti-symmetric matrices, and from the usual canonical form for the latter.

Let us also notice that non-zero wedge products $u\wedge v$ correspond to the $2$-dimensional subspaces $\langle u,v\rangle$ of $V$, where a change of basis of  the latter results in scaling $u\wedge v$ by the determinant of the corresponding invertible matrix. \end{remark}
 
%

\subsection{$3$-dimensional case}

In this subsection, we take $n=3$. Then
\[ \dim_{\mathbb{F}_p}(V) = 3,\,\ \dim_{\mathbb{F}_p}(\Lambda^2V) ={3\choose 2} = 3,\]
and the groups (\ref{Gpi}) that we construct are of order $p^{6}$. For $\pi$ of rank one, there are only two possibilities up to a change of basis.

\begin{prop}\label{rank one prop'} Let $\pi : V\rightarrow\Lambda^2V$ be a linear map of rank one.  

There exists a basis $v_1,v_2,v_3$ of $V$ such that
\[ v_2^\pi = v_3^\pi  =1\mbox{ and } v_1^\pi=\begin{cases}
(v_1\wedge v_2),\\
(v_2\wedge v_3).\\
\end{cases}\]
\end{prop}

\begin{proof}The hypothesis means that 
\[ \dim_{\mathbb{F}_p}(\ker(\pi))=2 \mbox{ and } \dim_{\mathbb{F}_p}(V^\pi)=1.\]
Let $\omega$ be a generator of $V^\pi$. We know that $\omega$ is expressible as a single wedge product because $n=3$. This implies that $\omega\in U\wedge U$ for some $2$-dimensional subspace $U$ of $V$. It is not hard to see that there exists a basis $v_1,v_2,v_3$ of $V$ such that
\[ v_1^\pi = (v_1\wedge v_2)\mbox{ with }U = \langle v_1,v_2\rangle \mbox{ and }\ker(\pi) =\langle v_2,v_3\rangle\]
when $U\cap \ker(\pi) = \langle v_2\rangle$ has dimension $1$, and
\[ v_1^\pi = (v_2\wedge v_3)\mbox{ with }U = \ker(\pi)= \langle v_2,v_3\rangle\]
when $U = \ker(\pi)=\langle v_2,v_3\rangle$ has dimension $2$.
\end{proof}

Next, let us compute $\Aut^c(\pi)$ for the two linear maps $\pi : V\rightarrow \Lambda^2V$ of rank one in Proposition \ref{rank one prop'}. We shall fix a basis $v_1,v_2,v_3$ of $V$ and write elements of $\GL(V)$ in matrix form
\[ \alpha = \begin{bmatrix}
a_{11} & a_{12} & a_{13} \\
a_{21} & a_{22} & a_{23} \\
a_{31} & a_{32} & a_{33} 
\end{bmatrix}\iff \begin{array}{c}
v_1^\alpha = v_1^{a_{11}} v_2^{a_{12}} v_3^{a_{13}},\\
v_2^\alpha=v_1^{a_{21}} v_2^{a_{22}} v_3^{a_{23}},\\
v_3^\alpha = v_1^{a_{31}} v_2^{a_{32}} v_3^{a_{33}}.
\end{array}\]
For the $\alpha\in \GL(V)$ such that $\pi\hat{\alpha} = \alpha\pi$, clearly $\ker(\pi)= \langle v_2,v_3\rangle$ is an invariant subspace of $\alpha$, which means that
\begin{equation}\label{alpha'}
 \alpha = \begin{bmatrix}
a_{11} & a_{12} & a_{13} \\
0 & a_{22} & a_{23} \\
0 & a_{32} & a_{33} 
\end{bmatrix}.
 \end{equation}
Note that then $a_{11}\neq 0$ since $\alpha$ is invertible, and we have $\alpha \in \Aut^c(\pi)$ exactly when $v_1^{\pi\hat{\alpha}}= v_1^{\alpha\pi}$. Moreover, in order for $v_1^{\pi\hat{\alpha}}= v_1^{\alpha\pi}$ to hold, the image $\langle v_1^\pi\rangle$ of $\pi$ must be invariant under $\hat{\alpha}$.

\begin{prop}\label{auto1'}Let $\pi :V \rightarrow\Lambda^2 V$ be the linear map defined by
\[ v_2^\pi = v_3^\pi=1\mbox{ and }v_1^\pi = (v_1\wedge v_2).\]
Then $\Aut^c(\pi) = P\rtimes Q$ is a semidirect product of the subgroups
\begin{align*}
P & = \left\{ 
\begin{bmatrix}
1 & b & 0 \\
0 & 1 & 0\\
0 & c & 1 
\end{bmatrix} : b,c,\in \mathbb{F}_p
\right\},\\
Q & = \left\{ 
\begin{bmatrix}
s & 0 & 0\\
0 & 1 & 0\\
0 & 0 & t
\end{bmatrix}: s,t\in\mathbb{F}_p^\times
\right\}.
\end{align*}
\end{prop}

\begin{proof}Let $\alpha\in \GL(V)$ be as in (\ref{alpha'}). For $\langle v_1^\pi\rangle$ to be invariant under $\hat{\alpha}$, necessarily $\langle v_1,v_2\rangle$ is invariant under $\alpha$. This means that
\[a_{13} = a_{23} =0,\]
and we see that
\begin{align*}
 v_1^{\pi\hat{\alpha}}= v_1^{\alpha\pi}&\iff v_1^{a_{11}}v_2^{a_{12}} \wedge v_2^{a_{22}} = (v_1\wedge v_2)^{a_{11}}\\& \iff a_{22} = 1.
 \end{align*}
It follows that $ \alpha \in \Aut^c(\pi)$ if and only if $\alpha$ has the form stated below. It is not hard to check that we have a surjective homomorphism
\begin{align*}
\Aut^c(\pi) & \twoheadrightarrow Q\\
\begin{bmatrix}
a_{11} & a_{12} & 0 \\
0 & 1& 0 \\
0 & a_{32} & a_{33} 
\end{bmatrix}&\mapsto  \begin{bmatrix}
a_{11} & 0 & 0\\
0 & 1 & 0 \\
0 &  0&a_{33}
\end{bmatrix}
\end{align*}
with kernel $P$, and we deduce that $\Aut^c(\pi) = P\rtimes Q$.
\end{proof}

\begin{prop}\label{auto2'}Let $\pi :V \rightarrow\Lambda^2 V$ be the linear map defined by
\[ v_2^\pi = v_3^\pi  =1\mbox{ and }v_1^\pi = (v_2\wedge v_3).\]
 Then $\Aut^c(\pi) = P\rtimes Q$ is a semidirect product of the subgroups
\begin{align*}
P & = \left\{ 
\begin{bmatrix}
1 & b & c \\
0 & 1 &0\\
 0 & 0 & 1
\end{bmatrix} : b,c\in \mathbb{F}_p
\right\},\\
Q & =\left\{ 
\begin{bmatrix}
|A| &  \begin{matrix} 0 & 0 \end{matrix}\\
 \begin{matrix} 0 \\ 0 \end{matrix} & A
\end{bmatrix}:  A \in \GL_2(\mathbb{F}_p)
\right\}.
\end{align*}
\end{prop}
\begin{proof}
Let $\alpha\in \GL(V)$ be as in (\ref{alpha'}). We have
\begin{align*}
 v_1^{\pi\hat{\alpha}}= v_1^{\alpha\pi} &\iff v_2^{a_{22}}v_3^{a_{23}}  \wedge v_2^{a_{32}}v_3^{a_{33}} = (v_2\wedge v_3)^{a_{11}}\\
 & \iff \begin{vmatrix}a_{22} &a_{23}\\ a_{32} & a_{33}\end{vmatrix} =a_{11}.\end{align*}
 It follows that $ \alpha \in \Aut^c(\pi)$ if and only if $\alpha$ has the form stated below. It is not hard to check that we have a surjective homomorphism
\begin{align*}
\Aut^c(\pi) & \twoheadrightarrow Q\\
\begin{bmatrix}
\lvert\begin{smallmatrix}
a_{22} & a_{23}\\a_{32} & a_{33}
\end{smallmatrix}\rvert&a_{12}& a_{13}\\
0 & a_{22} & a_{23}\\
0 & a_{32} & a_{33}
\end{bmatrix}&\mapsto \begin{bmatrix}
\lvert\begin{smallmatrix}
a_{22} & a_{23}\\a_{32} & a_{33}
\end{smallmatrix}\rvert& 0 & 0\\
0 & a_{22} & a_{23}\\
0 & a_{32} & a_{33}
\end{bmatrix}\end{align*}
with kernel $P$, and we deduce that $\Aut^c(\pi) = P\rtimes Q$.
\end{proof}

\subsection{4-dimensional case}
In this subsection, we take $n=4$. Then
\[ \dim_{\mathbb{F}_p}(V) = 4,\,\ \dim_{\mathbb{F}_p}(\Lambda^2V) ={4\choose 2} = 6,\]
and the groups (\ref{Gpi}) that we construct are of order $p^{10}$. For $\pi$ of rank one, there are only three possibilities up to a change of basis.

\begin{prop}\label{rank one prop}Let $\pi : V\rightarrow\Lambda^2V$ be a linear map of rank one. 

There exists a basis $v_1,v_2,v_3,v_4$ of $V$ such that
\[ v_2^\pi = v_3^\pi = v_4^\pi =1\mbox{ and } v_1^\pi=\begin{cases}
(v_1\wedge v_2),\\
(v_3\wedge v_4),\\
(v_1\wedge v_2)(v_3\wedge v_4).
\end{cases}\]
\end{prop}

\begin{proof}The hypothesis means that 
\[ \dim_{\mathbb{F}_p}(\ker(\pi))=3 \mbox{ and } \dim_{\mathbb{F}_p}(V^\pi)=1.\]
Let $\omega$ be a generator of $V^\pi$. We know that either $\omega$ is a single wedge product or a product of two wedge products because $n=4$.

Assume first that $\omega$ is a single wedge product. Then $\omega\in U\wedge U$ for some $2$-dimensional subspace $U$ of $V$. It is not hard to see that there exists a basis $v_1,v_2,v_3,v_4$ of $V$ such that
\[ v_1^\pi = (v_1\wedge v_2)\mbox{ with }U = \langle v_1,v_2\rangle \mbox{ and }\ker(\pi) =\langle v_2,v_3,v_4\rangle\]
when $U\cap \ker(\pi) = \langle v_2\rangle$ has dimension $1$, and
\[ v_1^\pi = (v_3\wedge v_4)\mbox{ with }U = \langle v_3,v_4\rangle\mbox{ and } \ker(\pi)= \langle v_2,v_3,v_4\rangle\]
when $U\cap \ker(\pi) = \langle v_3,v_4\rangle$ has dimension $2$.

Assume next that $\omega = \omega_1\omega_2$ is a product of two wedge products $\omega_1$ and $\omega_2$, where $\omega_1\in U_1\wedge U_1$ and $\omega_2\in U_2\wedge U_2$ for some $2$-dimensional subspaces $U_1$ and $U_2$ of $V$. We may assume that $U_1\cap U_2$ is trivial, for otherwise $\omega$ is expressible as a single wedge product. In the case that
\[ \dim_{\mathbb{F}_p}(U_1\cap \ker(\pi)) =  \dim_{\mathbb{F}_p}(U_2\cap \ker(\pi)) = 1,\]
we can find a basis $u_1,u_2$ of $U_1$ and a basis $u_3,u_4$ of $U_2$ such that
\[ \omega_1 = (u_1\wedge u_2),\,\ \omega_2 = (u_3\wedge u_4), \,\ \ker(\pi) = \langle u_2,u_4,u_1^{-1}u_3\rangle.\]
But notice that we can rewrite
\[ \omega = \omega_1\omega_2 = (u_1\wedge u_2u_4)(u_1^{-1}u_3\wedge u_4).\]
Thus, by changing the choice of $\omega_1,\omega_2$, we may assume that
\[ \dim_{\mathbb{F}_p}(U_1\cap \ker(\pi)) = 1\mbox{ and } \dim_{\mathbb{F}_p}(U_2\cap \ker(\pi)) = 2.\]
We then deduce that 
\[ v_1^\pi = (v_1\wedge v_2)(v_3\wedge v_4)\mbox{ with } 
\begin{cases}U_1 = \langle v_1,v_2\rangle\\
U_2 = \langle v_3,v_4\rangle
\end{cases}\hspace{-1em}
\mbox{ and } \ker(\pi) =\langle v_2,v_3,v_4\rangle
 \]
 for a suitable basis $v_1,v_2,v_3,v_4$ of $V$.
\end{proof}

Next, we compute $\Aut^c(\pi)$ for the three linear maps $\pi : V\rightarrow \Lambda^2V$ of rank one in Proposition \ref{rank one prop}. We shall fix a basis $v_1,v_2,v_3,v_4$ of $V$ and write elements of $\GL(V)$ in matrix form
\[ \alpha = \begin{bmatrix}
a_{11} & a_{12} & a_{13} & a_{14}\\
a_{21} & a_{22} & a_{23} & a_{24}\\
a_{31} & a_{32} & a_{33} & a_{34}\\
a_{41} & a_{42} & a_{43} & a_{44}
\end{bmatrix}\iff  
\begin{array}{c}
v_1^\alpha = v_1^{a_{11}} v_2^{a_{12}} v_3^{a_{13}} v_4^{a_{14}}, \\
v_2^\alpha  = v_1^{a_{21}} v_2^{a_{22}} v_3^{a_{23}} v_4^{a_{24}},\\
v_3^\alpha  = v_1^{a_{31}} v_2^{a_{32}} v_3^{a_{33}} v_4^{a_{34}},\\
v_4^\alpha  = v_1^{a_{41}} v_2^{a_{42}} v_3^{a_{43}} v_4^{a_{44}}.
\end{array}\]
For the $\alpha\in \GL(V)$ such that $\pi\hat{\alpha} = \alpha\pi$, clearly $\ker(\pi)= \langle v_2,v_3,v_4\rangle$ is an invariant subspace of $\alpha$, which means that
\begin{equation}\label{alpha}
 \alpha = \begin{bmatrix}
a_{11} & a_{12} & a_{13} &a_{14}\\
0 & a_{22} & a_{23} &a_{24}\\
0 & a_{32} & a_{33}  &a_{34}\\
0 & a_{42} & a_{43} & a_{44}
\end{bmatrix}.
 \end{equation}
 Note that then $a_{11}\neq 0$ since $\alpha$ is invertible, and we have $\alpha \in \Aut^c(\pi)$ exactly when $v_1^{\pi\hat{\alpha}}= v_1^{\alpha\pi}$. Moreover, in order for  $v_1^{\pi\hat{\alpha}}= v_1^{\alpha\pi}$ to hold, the image $\langle v_1^\pi\rangle$ of $\pi$ must be invariant under $\hat{\alpha}$.
 

\begin{prop}\label{auto1}Let $\pi :V \rightarrow\Lambda^2 V$ be the linear map defined by
\[ v_2^\pi = v_3^\pi=v_4^\pi =1\mbox{ and }v_1^\pi = (v_1\wedge v_2).\]
Then $\Aut^c(\pi) = P\rtimes Q$ is a semidirect product of the subgroups
\begin{align*}
P & = \left\{ 
\begin{bmatrix}
1 & b & 0 & 0 \\
0 & 1 & 0 & 0\\
0 & c & 1 & 0\\
0 & d & 0 & 1
\end{bmatrix} : b,c,d\in \mathbb{F}_p
\right\},\\
Q & = \left\{ 
\begin{bmatrix}
s & 0 & \begin{matrix} 0 & 0 \end{matrix}\\
0 & 1 & \begin{matrix} 0 & 0 \end{matrix}\\
\begin{matrix} 0 \\ 0 \end{matrix} & \begin{matrix} 0 \\ 0 \end{matrix} & A
\end{bmatrix}: s\in\mathbb{F}_p^\times,\, A \in \GL_2(\mathbb{F}_p)
\right\}.
\end{align*}
\end{prop}
\begin{proof}Let $\alpha\in \GL(V)$ be as in (\ref{alpha}). For $\langle v_1^\pi\rangle$ to be invariant under $\hat{\alpha}$, necessarily $\langle v_1,v_2\rangle$ is invariant under $\alpha$. This means that 
\[a_{13} = a_{14} = a_{23} = a_{24} =0,\]
and we see that
\begin{align*}
 v_1^{\pi\hat{\alpha}}= v_1^{\alpha\pi}&\iff v_1^{a_{11}}v_2^{a_{12}} \wedge v_2^{a_{22}} = (v_1\wedge v_2)^{a_{11}}\\& \iff a_{22} = 1.
 \end{align*}
 It follows that $ \alpha \in \Aut^c(\pi)$ if and only if $\alpha$ has the form stated below. It is not hard to check that we have a surjective homomorphism
\begin{align*}
\Aut^c(\pi) & \twoheadrightarrow Q\\
\begin{bmatrix}
a_{11} & a_{12} & 0 & 0\\
0 & 1 & 0 & 0 \\
0 & a_{32} & a_{33} & a_{34} \\
0 & a_{42} & a_{43} & a_{44}
\end{bmatrix}&\mapsto  \begin{bmatrix}
a_{11} & 0 & 0 & 0\\
0 & 1 & 0 & 0 \\
0 & 0& a_{33} & a_{34} \\
0 & 0 & a_{43} & a_{44}
\end{bmatrix}
\end{align*}
with kernel $P$, and we deduce that $\Aut^c(\pi) = P\rtimes Q$.
\end{proof}

\begin{prop}\label{auto2}Let $\pi :V \rightarrow\Lambda^2 V$ be the linear map defined by
\[ v_2^\pi = v_3^\pi=v_4^\pi =1\mbox{ and }v_1^\pi = (v_3\wedge v_4).\]
 Then $\Aut^c(\pi) = P\rtimes Q$ is a semidirect product of the subgroups
\begin{align*}
P & = \left\{ 
\begin{bmatrix}
1 & b & c & e \\
0 & 1 &d & f\\
0 & 0 & 1 & 0\\
0 & 0 & 0 & 1
\end{bmatrix} : b,c,d,e,f\in \mathbb{F}_p
\right\},\\
Q & =\left\{ 
\begin{bmatrix}
|A| & 0 & \begin{matrix} 0 & 0 \end{matrix}\\
0 & s & \begin{matrix} 0 & 0 \end{matrix}\\
\begin{matrix} 0 \\ 0 \end{matrix} & \begin{matrix} 0 \\ 0 \end{matrix} & A
\end{bmatrix}: s\in \mathbb{F}_p^\times,\, A \in \GL_2(\mathbb{F}_p)
\right\}.
\end{align*}
\end{prop}
\begin{proof}Let $\alpha\in \GL(V)$ be as in (\ref{alpha}). We have
\begin{align*}
v_1^{\pi\hat{\alpha}} =v_1^{\alpha\pi} & \iff   v_2^{a_{32}} v_3^{a_{33}} v_4^{a_{34}}\wedge v_2^{a_{42}} v_3^{a_{43}} v_4^{a_{44}} = (v_3\wedge v_4)^{a_{11}}\\
& \iff  \begin{vmatrix}
a_{32} & a_{33}\\
a_{42} & a_{43}
\end{vmatrix} =\begin{vmatrix} 
a_{32} & a_{34}\\
a_{42} & a_{44}
\end{vmatrix} =0,\, 
 a_{11}=\begin{vmatrix}
a_{33} & a_{34}\\
a_{43} & a_{44}
\end{vmatrix} .
\end{align*}
Observe that then $(a_{32},a_{42}) = (0,0)$ must hold, for otherwise $(a_{33},a_{43})$ and $(a_{34},a_{44})$ would be scalar multiples of each other by the first relation, and so $a_{11} = 0$ by the second relation. It follows that $\alpha\in \Aut^c(\pi)$ if and only if $\alpha$ has the form stated below. It is not hard to check that we have a surjective homomorphism
\begin{align*}
\Aut^c(\pi) & \twoheadrightarrow Q\\
\begin{bmatrix}
\lvert\begin{smallmatrix}
a_{33} & a_{34}\\
a_{43} & a_{44}
\end{smallmatrix}\rvert & a_{12} & a_{13} & a_{14}\\
0 & a_{22} & a_{23} & a_{24} \\
0 & 0 & a_{33} & a_{34} \\
0 & 0 & a_{43} & a_{44}
\end{bmatrix}&\mapsto \begin{bmatrix}
\lvert\begin{smallmatrix}
a_{33} & a_{34}\\
a_{43} & a_{44}
\end{smallmatrix}\rvert & 0 & 0& 0\\
0 & a_{22} & 0& 0 \\
0 & 0 & a_{33} & a_{34} \\
0 & 0 & a_{43} & a_{44}
\end{bmatrix}
\end{align*}
with kernel $P$, and we deduce that $\Aut^c(\pi) = P\rtimes Q$.
\end{proof}

\begin{prop}\label{auto3}Let $\pi :V \rightarrow\Lambda^2 V$ be the linear map defined by
\[ v_2^\pi = v_3^\pi = v_4^\pi = 1 \mbox{ and }v_1^\pi = (v_1\wedge v_2)(v_3\wedge v_4).\]
Then $\Aut^c(\pi) = P\rtimes Q$ is a semidirect product of the subgroups
\begin{align*}
P & = \left\{ 
\begin{bmatrix}
1 & b & -d & c\\
0 & 1 & 0 & 0\\
0 & c & 1 & 0\\
0 & d & 0 & 1
\end{bmatrix} : b,c,d\in \mathbb{F}_p
\right\},\\
Q & =\left\{ 
\begin{bmatrix}
|A| & 0 & \begin{matrix} 0 & 0 \end{matrix}\\
0 & 1 & \begin{matrix} 0 & 0 \end{matrix}\\
\begin{matrix} 0 \\ 0 \end{matrix} & \begin{matrix} 0 \\ 0 \end{matrix} & A
\end{bmatrix}: A \in \GL_2(\mathbb{F}_p)
\right\}.
\end{align*}
\end{prop}
\begin{proof}Let $\alpha\in \GL(V)$ be as in (\ref{alpha}). For $\langle v_1^\pi\rangle$ to be invariant under $\hat{\alpha}$, necessarily $\langle v_2\rangle$ is invariant under $\alpha$. This means that 
\[a_{23} = a_{24} =0,\]
and so we have
\begin{align*}
v_1^{\pi\hat{\alpha}}&= (v_1^{a_{11}} v_2^{a_{12}} v_3^{a_{13}} v_4^{a_{14}}\wedge  v_2^{a_{22}})(v_2^{a_{32}} v_3^{a_{33}} v_4^{a_{34}}\wedge v_2^{a_{42}} v_3^{a_{43}} v_4^{a_{44}} ),\\
v_1^{\alpha\pi}&=(v_1\wedge v_2)^{a_{11}}(v_3\wedge v_4)^{a_{11}}.
\end{align*}
We then see that $v_1^{\pi\hat{\alpha}} = v_1^{\alpha\pi}$ holds if and only if
\[a_{22} = 1,\, -a_{13}+\begin{vmatrix}
a_{32} & a_{33}\\
a_{42} & a_{43}
\end{vmatrix} = -a_{14} +\begin{vmatrix}
a_{32} & a_{34}\\
a_{42} & a_{44}
\end{vmatrix} =0,\,
 \begin{vmatrix}
a_{33} & a_{34}\\
a_{43} & a_{44}
\end{vmatrix} = a_{11}.\]
It follows that $\alpha\in \Aut^c(\pi)$ if and only if $\alpha$ has the form stated below. It is not hard to check that we have a surjective homomorphism
\begin{align*}
\Aut^c(\pi) & \twoheadrightarrow Q\\
\begin{bmatrix}
\lvert\begin{smallmatrix}
a_{33} & a_{34}\\
a_{43} & a_{44}
\end{smallmatrix}\rvert& a_{12} & \lvert\begin{smallmatrix}
a_{32} & a_{33}\\
a_{42} & a_{43}
\end{smallmatrix}\rvert & \lvert\begin{smallmatrix}
a_{32} & a_{34}\\
a_{42} & a_{44}
\end{smallmatrix}\rvert\\
0 & 1 & 0 & 0 \\
0 & a_{32} & a_{33} & a_{34} \\
0 & a_{42} & a_{43} & a_{44}
\end{bmatrix}&\mapsto \begin{bmatrix}
\lvert\begin{smallmatrix}
a_{33} & a_{34}\\
a_{43} & a_{44}
\end{smallmatrix}\rvert&0 &0 &0\\
0 & 1 & 0 & 0 \\
0 & 0 & a_{33} & a_{34} \\
0 & 0 & a_{43} & a_{44}
\end{bmatrix}\end{align*}
with kernel $P$, and we deduce that $\Aut^c(\pi) = P\rtimes Q$.
\end{proof}

\section{Proof of the main theorem}

Let $p$ be an odd prime and let $V$ be an $n$-dimensional vector space over $\mathbb{F}_p$ with basis $v_1,v_2,\dots,v_n$. The groups $G$ in (a),(b) are precisely those in (\ref{Gpi}), with $n=3$, associated to the linear maps
\begin{enumerate}[label = (\alph*)]
\item $\pi : V\rightarrow \Lambda^2V;\,$ $v_2^\pi = v_3^\pi  =1 $ and $v_1^\pi = (v_1\wedge v_2)$,
\item $\pi : V\rightarrow \Lambda^2V;\,$ $v_2^\pi = v_3^\pi =1 $ and $v_1^\pi = (v_2\wedge v_3)$.
\end{enumerate}
Similarly, the groups $G$ in (c),(d),(e) are precisely those in  (\ref{Gpi}), with $n=4$, associated to the linear maps
\begin{enumerate}[label = (\alph*)]\setcounter{enumi}{+2}
\item $\pi : V\rightarrow \Lambda^2V;\,$ $v_2^\pi = v_3^\pi = v_4^\pi =1 $ and $v_1^\pi = (v_1\wedge v_2)$,
\item $\pi : V\rightarrow \Lambda^2V;\,$ $v_2^\pi = v_3^\pi = v_4^\pi =1 $ and $v_1^\pi = (v_3\wedge v_4)$,
\item $\pi : V\rightarrow \Lambda^2V;\,$ $v_2^\pi = v_3^\pi = v_4^\pi =1 $ and $v_1^\pi = (v_1\wedge v_2)(v_3\wedge v_4)$.
\end{enumerate}
By Propositions \ref{rank one prop'} and \ref{rank one prop}, for $n=3,4$ and up to a change of basis, these are the only linear maps $\pi$ of rank one.

In this section, let us take $n=3,4$ and the symbol $\pi$ denotes one of the five linear maps above. As explained in Section \ref{group section}, we may identify
\[ G/G' = V\mbox{ and }G'=\Lambda^2V.\]
Moreover, we have a natural isomorphism
\[ \Aut^c(G) \simeq \Aut^c(\pi).\]
With these identifications, we may rephrase (\ref{Delta2}) as
\begin{equation}\label{Delta3}
\Delta(u^\alpha,v^\alpha) = \Delta(u,v)^{\hat{\alpha}}
\end{equation}
for all $u,v\in V$ and $\alpha\in\Aut^c(\pi)$. The $S$ and $S'$ in Section \ref{bilinear form sec} become
\begin{align*}
S &=  \{\mbox{symmetric bilinear $\Delta :V\times V\rightarrow \Lambda^2V$ satisfying (\ref{Delta3})}\}\\
S' &= \{\mbox{anti-symmetric bilinear  $\Delta :V\times V\rightarrow \Lambda^2V$ satisfying (\ref{Delta3})}\}
\end{align*}
in the current setting. The group $\Aut^c(\pi)$ was computed in Section \ref{group section}. Let $P$ and $Q$ denote the subgroups defined there. Then, we have
\[\Aut^c(\pi) = P\rtimes Q.\]
We shall also make the following assumption.

\begin{assume}Assume that $p\geq 5$ in the cases (a),(c),(e).
\end{assume}

We first show that the groups $G$ in question satisfy Assumption \ref{assumption} so that the discussion thereafter applies.  
\begin{lemma}\label{gamma lemma}
Let $\gamma : V\rightarrow\Aut^c(\pi)$ be an $\Aut^c(\pi)$-equivariant homomorphism and let $1\leq i,j\leq n$. Suppose that
\begin{enumerate}[label = $(\arabic*)$]
\item $\gamma(v_i)=1$,
\item $v_i^\alpha = v_iv_j$ for some $\alpha\in \Aut^c(\pi)$.
\end{enumerate}
Then $\gamma(v_j)=1$ also holds.
\end{lemma}

\begin{proof}Indeed, we have
\[ 1 = \gamma(v_i)^\alpha = \gamma(v_i^\alpha) = \gamma(v_i)\gamma(v_j) = \gamma(v_j)\]
by the hypotheses.
\end{proof}

\begin{prop}
  There is no non-trivial $\Aut^c(\pi)$-equivariant homomorphism from $V$ to $\Aut^c(\pi)$.
\end{prop}

\begin{proof}Let $\gamma : V\rightarrow\Aut^c(\pi)$ be an $\Aut^c(\pi)$-equivariant homomorphism and observe that $\gamma(V)$ must be a normal $p$-subgroup of $\Aut^c(\pi)$. But
 \[ Q \simeq \begin{cases}
\mathbb{F}_p^\times\times \mathbb{F}_p^\times &\mbox{in case (a)}\\
\GL_2(\mathbb{F}_p)&\mbox{in cases (b) and (e)}\\
\mathbb{F}_p^\times \times \GL_2(\mathbb{F}_p)&\mbox{in cases (c) and (d)}
\end{cases}\]
has no non-trivial normal $p$-subgroup. Since $\Aut^c(\pi) = P\rtimes Q$, we see that $\gamma(V)$ must lie inside $P$. We now deal with each case separately.
\begin{enumerate}[label=(\alph*), wide=0pt]
\item It is clear from Proposition \ref{auto1'} that
\[ v_1^{\alpha_{12}} = v_1v_2\]
for some $\alpha_{12} \in P$, and so it is enough to show that $\gamma(v_1)=\gamma(v_3)=1$ by Lemma \ref{gamma lemma}. Let us put
\[ \gamma(v_1) = \begin{bmatrix}
1 & b_1 & 0 \\
0 & 1 & 0 \\
0 & c_1 & 1
\end{bmatrix}\mbox{ and }\gamma(v_3)= \begin{bmatrix}
1 & b_3 & 0 \\
0 & 1 & 0 \\
0 & c_3 & 1
\end{bmatrix}.\]
From $\gamma(v_1^\alpha) = \gamma(v_1)^\alpha$ for $\alpha\in Q$ of the shape
\[ \alpha = \begin{bmatrix}s & 0 & 0 \\
0 & 1 & 0\\
0 & 0 & s\end{bmatrix} \mbox{ with } s\in \mathbb{F}_p^\times,\]
we get that $\gamma(v_1^\alpha) = \gamma(v_1)^s$ and
\[  \begin{bmatrix}
1 & sb_1 & 0 \\
0 & 1 & 0 \\
0 & sc_1 & 1
\end{bmatrix}= \begin{bmatrix}
1 & s^{-1}b_1 & 0 \\
0 & 1 & 0 \\
0 & s^{-1}c_1 & 1
\end{bmatrix}.\]
Since $p\geq 5$, there exists $s\in \mathbb{F}_p^\times$ with $s^2\neq 1$, and so $b_1=c_1=0$. We may obtain $b_3 = c_3 =0$ by the exact same calculation.
\item It is clear from Proposition \ref{auto2'} that
\[ v_1^{\alpha_{12}} = v_1v_2\mbox{ and } v_1^{\alpha_{13}} = v_1v_3\]
for some $\alpha_{12},\alpha_{13} \in P$, so it suffices to show that $\gamma(v_1)=1$ by Lemma \ref{gamma lemma}. Let us put
\[ \gamma(v_1) = \begin{bmatrix}1 & b_1 & c_1\\0& 1 & 0 \\ 0 & 0 & 1\end{bmatrix}.\]
From $\gamma(v_1^\alpha) = \gamma(v_1)^\alpha$ for $\alpha\in Q$ of the shape
\[\begin{bmatrix}
1 & 0 & 0\\
0 & s & 0\\
0 & 0 &s^{-1}
\end{bmatrix}\mbox{ with }s\in\mathbb{F}_p^\times,\]
we get that $\gamma(v_1^\alpha) = \gamma(v_1)$ and
\[  \begin{bmatrix}1 & b_1 & c_1\\0& 1 & 0 \\ 0 & 0 & 1\end{bmatrix}
=  \begin{bmatrix}1 & sb_1 & s^{-1}c_1\\0& 1 & 0 \\ 0 & 0 & 1\end{bmatrix}.\]
This yields $b_1=c_1=0$.
\item It is clear from Proposition \ref{auto1} that
\[ v_1^{\alpha_{12}} = v_1v_2\mbox{ and }v_3^{\alpha_{34}} = v_3v_4\]
for some $\alpha_{12}\in P, \alpha_{34}\in Q$, so it suffices to show that $\gamma(v_1)=\gamma(v_3)=1$ by Lemma \ref{gamma lemma}. Let us put
\[ \gamma(v_1) = \begin{bmatrix}
1 & b_1 & 0 & 0\\
0 & 1 & 0 & 0\\
0 & c_1 & 1 & 0\\
0 & d_1 & 0 & 1
\end{bmatrix}\mbox{ and }
 \gamma(v_3) = \begin{bmatrix}
1 & b_3 & 0 & 0\\
0 & 1 & 0 & 0\\
0 & c_3 & 1 & 0\\
0 & d_3 & 0 & 1
\end{bmatrix}.\]
From $\gamma(v_1^\alpha) = \gamma(v_1)^\alpha$ for $\alpha\in  Q$ of the shape
\[ \alpha =\begin{bmatrix}
s & 0 & 0 &0\\
0 & 1 & 0 & 0\\
0 & 0 & s & 0\\
0 & 0 & 0 & s
\end{bmatrix} \mbox{ with } s\in \mathbb{F}_p^\times,\]
we get that $\gamma(v_1^\alpha) = \gamma(v_1)^s$ and
\[ \begin{bmatrix}
1 & sb_1 & 0 & 0\\
0 & 1 & 0 & 0\\
0 & sc_1 &1 & 0\\
0 & sd_1 & 0 & 1
\end{bmatrix} = \begin{bmatrix}
1 & s^{-1}b_1 & 0 & 0\\
0 & 1 & 0 & 0\\
0 & s^{-1}c_1 &1 & 0\\
0 & s^{-1}d_1 & 0 & 1
\end{bmatrix} .\]
Since $p\geq 5$, there exists $s\in \mathbb{F}_p^\times$ with $s^2\neq 1$, and so $b_1=c_1=d_1=0$. We may obtain $b_3=c_3=d_3=0$ by the exact same calculation.
\item It is clear from Proposition \ref{auto2} that
\[ v_1^{\alpha_{12}} = v_1v_2,\,\ v_1^{\alpha_{13}} = v_1v_3,\,\ v_1^{\alpha_{14}} = v_1v_4.\]
for some $\alpha_{12},\alpha_{13},\alpha_{14}\in P$, and so it suffices to show that $\gamma(v_1)=1$ by Lemma \ref{gamma lemma}. Let us put
\[ \gamma(v_1) = \begin{bmatrix}
1 & b_1 & c_1 & e_1\\
0 & 1 & d_1 & f_1\\
0 & 0 & 1 & 0\\
0 & 0 & 0 & 1
\end{bmatrix}.\]
From $\gamma(v_1^\alpha) = \gamma(v_1)^\alpha$ for $\alpha\in \Aut^c(\pi)$ of the shape
\[ \alpha =\begin{bmatrix}
1 & 0 & 0 & 0\\
0 & 1 & g & 0\\
0 & 0 & s & 0\\
0 & 0 & 0 & s^{-1}
\end{bmatrix} \mbox{ with  } s\in \mathbb{F}_p^\times\mbox{ and }g\in \mathbb{F}_p,\]
we get that $\gamma(v_1^\alpha) = \gamma(v_1)$ and
\[ \begin{bmatrix}
1 & b_1 & c_1 & e_1\\
0 & 1 & d_1 & f_1\\
0 & 0 &1 & 0\\
0 & 0 & 0 & 1
\end{bmatrix} = \begin{bmatrix}
1 & b_1 & gb_1 + sc_1 & s^{-1}e_1\\
0 & 1 & sd_1 & s^{-1}f_1\\
0 & 0 & 1 & 0\\
0 & 0 & 0 &1 
\end{bmatrix}.\]
This yields $b_1 = c_1 = d_1=e_1=f_1=0$. 
\item It is clear from Proposition \ref{auto3} that
\[v_1^{\alpha_{12}} = v_1v_2,\,\ v_1^{\alpha_{13}} = v_1v_3,\,\ v_1^{\alpha_{14}} = v_1v_4\]
for some $\alpha_{12},\alpha_{13},\alpha_{14}\in P$, and so it suffices to show that $\gamma(v_1)=1$ by Lemma \ref{gamma lemma}. Let us put
\[ \gamma(v_1) = \begin{bmatrix}
1 & b_1 & -d_1 & c_1\\
0 & 1 & 0 & 0\\
0 & c_1 & 1 & 0\\
0 & d_1 & 0 & 1
\end{bmatrix}.\]
From $\gamma(v_1^\alpha) = \gamma(v_1)^\alpha$ for $\alpha\in Q$ of the shape
\[ \alpha =\begin{bmatrix}
s& 0 & 0 & 0\\
0 & 1 & 0 & 0\\
0 & 0 & s & 0\\
0 & 0 & 0 &1
\end{bmatrix} \mbox{ with } s\in \mathbb{F}_p^\times,\]
we get that $\gamma(v_1^\alpha) = \gamma(v_1)^{s}$ and
\[\begin{bmatrix}
1 & sb_1 & -sd_1 & sc_1\\
0 & 1 & 0 & 0\\
0 & sc_1 & 1 & 0\\
0 & sd_1 & 0 & 1
\end{bmatrix}= \begin{bmatrix}
1 & s^{-1}b_1 & -d_1 & s^{-1}c_1\\
0 & 1 & 0 & 0\\
0 & s^{-1}c_1 & 1 & 0\\
0 & d_1 & 0 & 1\end{bmatrix}.\]
This implies that $d_1=0$. Since $p\geq 5$, there exists $s\in \mathbb{F}_p^\times$ with $s^2\neq 1$, and we see that $b_1=c_1=0$ as well.
\end{enumerate}
In all cases, we have shown that $\gamma$ is trivial.
 \end{proof}
 
 Therefore, we may apply Theorem \ref{pre thm} to obtain
 \begin{equation}\label{T(G)} T(G) \simeq S \rtimes \res(\mathcal{S}').\end{equation}
It remains to determine the structure of $S$ and $\res(\mathcal{S}')$.

 \subsection{A module-theoretic approach} 
 
Observe that by the universal property of $S^2V$, the symmetric square of $V$, there is a natural correspondence between
\begin{itemize}
\item symmetric bilinear forms $V\times V\rightarrow\Lambda^2V$,
\item linear maps $S^2V\rightarrow \Lambda^2V$.
\end{itemize}
Similarly, there is a natural correspondence between
\begin{itemize}
\item anti-symmetric bilinear forms $V\times V\rightarrow\Lambda^2V$,
\item linear maps $\Lambda^2V\rightarrow \Lambda^2V$.
\end{itemize}
Since we are writing addition in $V$ multiplicatively, let us denote multiplication in $S^2V$ by $*$ to avoid confusion. Then, both $S^2V$ and $\Lambda^2V$ are naturally $\Aut^c(\pi)$-modules via the action
\[ (u* v)^{\alpha} = u^\alpha * v^\alpha\mbox{ and }(u\wedge v)^\alpha = u^\alpha \wedge v^\alpha\]
for all $u,v\in V$ and $\alpha\in \Aut^c(\pi)$. Taking (\ref{Delta3}) into account, it follows that elements of $S$ and $S'$, respectively, correspond to $\Aut^c(\pi)$-module homomorphisms $S^2V\rightarrow \Lambda^2V$ and $\Lambda^2V\rightarrow\Lambda^2V$.

Let us first restrict the action to $Q$. An $\Aut^c(\pi)$-module homomorphism is in particular a $Q$-module homomorphism. The latter is easier to understand because matrices in $Q$ are all block diagonal, and so we easily see that both $S^2V$ and $\Lambda^2V$, as $Q$-modules, are decomposable as a direct sum of irreducible submodules. In the tables below, we list a basis for each irreducible component, and we indicate the action of an arbitrary $\alpha\in Q$ in matrix form with respect to the given basis. Here
\[ \alpha = \begin{bmatrix} s & 0 & 0 \\ 0 & 1 & 0 \\ 0 & 0 &t\end{bmatrix},\begin{bmatrix}
|A| &  \begin{matrix} 0 & 0 \end{matrix}\\
 \begin{matrix} 0 \\ 0 \end{matrix} & A
\end{bmatrix}\]
in cases (a),(b), respectively, while 
\[ \alpha = \begin{bmatrix}
s & 0 & \begin{matrix} 0 & 0 \end{matrix}\\
0 & 1 & \begin{matrix} 0 & 0 \end{matrix}\\
\begin{matrix} 0 \\ 0 \end{matrix} & \begin{matrix} 0 \\ 0 \end{matrix} & A
\end{bmatrix},\begin{bmatrix}
|A| & 0 & \begin{matrix} 0 & 0 \end{matrix}\\
0 & s & \begin{matrix} 0 & 0 \end{matrix}\\
\begin{matrix} 0 \\ 0 \end{matrix} & \begin{matrix} 0 \\ 0 \end{matrix} & A
\end{bmatrix},\begin{bmatrix}
|A| & 0 & \begin{matrix} 0 & 0 \end{matrix}\\
0 & 1 & \begin{matrix} 0 & 0 \end{matrix}\\
\begin{matrix} 0 \\ 0 \end{matrix} & \begin{matrix} 0 \\ 0 \end{matrix} & A
\end{bmatrix}\]
in cases (c),(d),(e), respectively. The variables $s,t$ here range over $\mathbb{F}_p^\times$, and $A$ ranges over $\GL_2(\mathbb{F}_p)$.

 \begingroup
\setlength{\tabcolsep}{10pt} 
\renewcommand{\arraystretch}{1.15}
 \begin{center}
 \begin{longtable}{ |c|c|}
 \hline
\multicolumn{2}{|c|}{\textbf{Components of $S^2V$ in case (a)}} \\
\hline
 Basis & Action of $\alpha\in Q$\\ \hline
 $v_1*v_1 $ & $s^2$ \\ 
 $v_1*v_2$ & $s$ \\ 
 $v_1*v_3$ & $st$ \\ 
 $v_2*v_2$ & $1$\\
 $v_2*v_3$ & $t$ \\
 $v_3*v_3$ & $t^2$\\
\hline\hline
\multicolumn{2}{|c|}{\textbf{Components of $\Lambda^2V$ in case (a)}}\\
\hline
 Basis & Action of $\alpha\in Q$ \\ \hline
 $v_1\wedge v_2 $ & $s$\\ 
 $v_1\wedge v_3$ & $st$  \\ 
 $v_2\wedge v_3$ & $t$\\ 
\hline
\end{longtable} 
  \begin{longtable}{ |c|c|}
 \hline
\multicolumn{2}{|c|}{\textbf{Components of $S^2V$ in case (b)}}\\
\hline
 Basis & Action of $\alpha\in Q$  \\ \hline
 $v_1*v_1 $ & $|A|^2$ \\ 
 $v_1*v_2,v_1*v_3$ & $|A|A$  \\ 
 $v_2*v_2, v_2*v_3,v_3*v_3$ & omitted  \\ 
\hline
\hline
\multicolumn{2}{|c|}{\textbf{Components of $\Lambda^2V$ in case (b)}}\\
\hline
 Basis & Action of $\alpha\in Q$ \\ \hline
 $v_1\wedge v_2 ,v_1\wedge v_3$ & $|A|A$  \\ 
 $v_2\wedge v_3$ & $|A|$  \\ 
\hline
\end{longtable}
 \begin{longtable}{ |c|c| }
 \hline
\multicolumn{2}{|c|}{\textbf{Components of $S^2V$ in case (c)}}\\
\hline
 Basis & Action of $\alpha\in Q$ \\ \hline
 $v_1*v_1 $ & $s^2$ \\ 
 $v_1*v_2$ & $s$ \\ 
 $v_1*v_3,v_1*v_4$ & $sA$  \\ 
 $v_2*v_2$ & $1$ \\
 $v_2*v_3,v_2*v_4$ & $A$ \\
 $v_3*v_3,v_3*v_4,v_4*v_4$ & omitted \\
\hline
\hline
\multicolumn{2}{|c|}{\textbf{Components of $\Lambda^2V$ in case (c)}} \\
\hline
 Basis & Action of $\alpha\in Q$  \\ \hline
 $v_1\wedge v_2 $ & $s$  \\ 
 $v_1\wedge v_3,v_1\wedge v_4$ & $sA$  \\ 
 $v_2\wedge v_3,v_2 \wedge v_4$ & $A$  \\ 
 $v_3\wedge v_4$ & $|A|$ \\
\hline
\end{longtable}
 \begin{longtable}{ |c|c| }
 \hline
\multicolumn{2}{|c|}{\textbf{Components of $S^2V$ in case (d)}}\\
\hline 
Basis & Action of $\alpha\in Q$ \\ \hline
 $v_1*v_1 $ & $|A|^2$ \\ 
 $v_1*v_2$ & $s|A|$  \\ 
 $v_1*v_3,v_1*v_4$ & $|A|A$  \\ 
 $v_2*v_2$ & $s^2$\\
 $v_2*v_3,v_2*v_4$ & $sA$  \\
 $v_3*v_3,v_3*v_4,v_4*v_4$ & omitted \\
\hline
\hline
\multicolumn{2}{|c|}{\textbf{Components of $\Lambda^2V$ in case (d)}}\\
\hline 
Basis & Action of $\alpha\in Q$  \\ \hline
 $v_1\wedge v_2 $ & $s|A|$ \\ 
 $v_1\wedge v_3,v_1\wedge v_4$ & $|A|A$  \\ 
 $v_2\wedge v_3,v_2 \wedge v_4$ & $sA$  \\ 
 $v_3\wedge v_4$ & $|A|$ \\
\hline
\end{longtable}
 \begin{longtable}{ |c|c| }
 \hline
\multicolumn{2}{|c|}{\textbf{Components of $S^2V$ in (e)}}\\
\hline 
Basis & Action of $\alpha\in Q$ \\ \hline
 $v_1*v_1 $ & $|A|^2$ \\ 
 $v_1*v_2$ & $|A|$  \\ 
 $v_1*v_3,v_1*v_4$ & $|A|A$  \\ 
 $v_2*v_2$ & $1$ \\
 $v_2*v_3,v_2*v_4$ & $A$  \\
 $v_3*v_3,v_3*v_4,v_4*v_4$ & omitted \\
\hline
\hline
\multicolumn{2}{|c|}{\textbf{Components of $\Lambda^2V$ in case (e)}}\\
\hline 
Basis & Action of $\alpha\in Q$  \\ \hline
 $v_1\wedge v_2 $ & $|A|$ \\ 
 $v_1\wedge v_3,v_1\wedge v_4$ & $|A|A$  \\ 
 $v_2\wedge v_3,v_2 \wedge v_4$ & $A$  \\ 
 $v_3\wedge v_4$ & $|A|$  \\
\hline
\end{longtable} 
\end{center} 
\endgroup

Under a $Q$-module homomorphism, an irreducible component of the domain either lies in the kernel or gets mapped to an isomorphic irreducible component of the codomain. From the stated action of $Q$, we can easily compare the isomorphism classes of the irreducible components of $S^2V$ and $\Lambda^2V$. Note that the omitted action does not matter because $\Lambda^2V$ does not have any $3$-dimensional irreducible component. The next two propositions are then immediate. 
 
 \begin{prop}\label{prelim prop sym}For any $\Delta\in S$, the following holds.
 \begin{enumerate}[label= $(\arabic*)$]
 \item In case (a), we have
\begin{align*}\Delta(v_1,v_1)&=1,\\
\Delta(v_2,v_2) &=1,\\
 \Delta(v_3,v_3)&=1.
\end{align*}
\item In case (b), we have
\begin{align*}
\Delta(v_1,v_1)& = 1,\\
\Delta(v_2,v_2) &= \Delta(v_2,v_3) =\Delta(v_3,v_3)=1.
\end{align*}
\item In cases (c),(d), and (e), we have
\begin{align*}
\Delta(v_1,v_1) &=1,\\
 \Delta(v_2,v_2) &=1,\\
 \Delta(v_3,v_3) &= \Delta(v_3,v_4)=\Delta(v_4,v_4) =1.
 \end{align*}
 \end{enumerate}
 \end{prop}
 
 \begin{prop}\label{prelim prop anti}
For any $\Delta\in S'$, the following holds.
 \begin{enumerate}[label= $(\arabic*)$]
\item In case (a), we have
\begin{align*}
\Delta(v_1,v_2) & \in \langle v_1\wedge v_2\rangle,\\
\Delta(v_1,v_3) & \in \langle v_1\wedge v_3\rangle,\\
\Delta(v_2,v_3) & \in \langle v_2\wedge v_3\rangle.
\end{align*}
\item In case (b), we have
\begin{align*}
\Delta(v_1,v_2),\Delta(v_1,v_3)& \in \langle v_1\wedge v_2,v_1\wedge v_3\rangle,\\
\Delta(v_2,v_3) & \in \langle v_2\wedge v_3\rangle.
\end{align*}
\item In cases (c) and (d), we have
\begin{align*}
 \Delta(v_1,v_2) & \in \langle v_1\wedge v_2\rangle,\\\
 \Delta(v_1,v_3),\Delta(v_1,v_4) & \in \langle v_1\wedge v_3, v_1\wedge v_4 \rangle,\\
 \Delta(v_2,v_3),\Delta(v_2,v_4) & \in \langle v_2\wedge v_3, v_2\wedge v_4 \rangle, \\
 \Delta(v_3,v_4) & \in \langle v_3\wedge v_4\rangle.
\end{align*}
\item In case (e), we have
\begin{align*}
 \Delta(v_1,v_2),\Delta(v_3,v_4) & \in \langle v_1\wedge v_2, v_3\wedge v_4\rangle,\\
 \Delta(v_1,v_3),\Delta(v_1,v_4) & \in \langle v_1\wedge v_3, v_1\wedge v_4 \rangle,\\
 \Delta(v_2,v_3),\Delta(v_2,v_4) & \in \langle v_2\wedge v_3, v_2\wedge v_4 \rangle.
\end{align*}
\end{enumerate}  
 \end{prop}
  
We may refine parts of Proposition \ref{prelim prop anti} as follows.

 \begin{prop}\label{scalar prop} For any $\Delta\in S'$, the following holds.
 \begin{enumerate}[label= $(\arabic*)$]
 \item In case (b), there exists $\lambda\in\mathbb{F}_p$ such that
 \[ \begin{cases}
 \Delta(v_1,v_2) = (v_1\wedge v_2)^\lambda,\\
 \Delta(v_1,v_3) = (v_1\wedge v_3)^\lambda.
 \end{cases}\]
 \item In cases (c),(d), and (e), there exist $\lambda_1,\lambda_2\in \mathbb{F}_p$ such that
\[\begin{cases}
\Delta(v_1,v_3) = (v_1\wedge v_3)^{\lambda_1} \\
\Delta(v_1,v_4) = (v_1\wedge v_4)^{\lambda_1}
\end{cases}\,\
\begin{cases}
\Delta(v_2,v_3) = (v_2\wedge v_3)^{\lambda_2},\\
\Delta(v_2,v_4) = (v_2\wedge v_4)^{\lambda_2}.
\end{cases}\]
 \end{enumerate}
  \end{prop}
 
 \begin{proof} Consider case (b). We know from Proposition \ref{prelim prop anti} that $\Delta$ has to induce a $Q$-module endomorphism 
 \[ \delta : \langle v_1\wedge v_2,v_1\wedge v_3\rangle \rightarrow \langle v_1\wedge v_2,v_1\wedge v_3 \rangle.\]
If $\delta$ is trivial, then simply take $\lambda=0$. If $\delta$ is non-trivial, then it has to be invertible because $\langle v_1\wedge v_2,v_1\wedge v_3\rangle$ is irreducible. Say $\delta$ is given by the matrix $M\in \GL_2(\mathbb{F}_p)$. But $M$ must commute with the action of $Q$ and observe that $Q$ restricts to an $\SL_2(\mathbb{F}_p)$-action on $\langle v_1\wedge v_2,v_1\wedge v_3\rangle$.
Since the only matrices that centralize $\SL_2(\mathbb{F}_p)$ are the scalar multiples of the identity, it follows that $M = \left[\begin{smallmatrix} \lambda& 0\\ 0 & \lambda\end{smallmatrix}\right]$
for some $\lambda\in\mathbb{F}_p^\times$. This proves (1), and the same argument may be applied to prove (2).\end{proof}

\subsection{Computation of $S$ and $S'$} We shall now compute $S$ and $S'$ by taking the action of $P$ into account.

First, notice that a symmetric bilinear form $\Delta : V\times V\rightarrow \Lambda^2V$ is uniquely determined by
  \[ \Delta(v_i,v_j)\mbox{ for }1\leq i \leq j \leq n.\]
The next observation shall also be useful.

\begin{lemma}\label{sym lemma}Let $\Delta \in S$ and let $1\leq i, j \leq n$. If
\begin{enumerate}[label = $(\arabic*)$]
\item $\Delta(v_i,v_i) = \Delta(v_j,v_j)=1$,
\item  $v_i^{\alpha }= v_iv_j$ for some $\alpha\in \Aut^c(\pi)$,
\end{enumerate}
then $\Delta(v_i,v_j) =\Delta(v_j,v_i)= 1$ also holds.
\end{lemma}

\begin{proof}By the hypothesis and the condition (\ref{Delta3}), we have
\begin{align*}
1 & = \Delta(v_i,v_i)^{\hat{\alpha}}\\
& = \Delta(v_i^{\alpha} ,v_i^{\alpha})\\
& = \Delta(v_iv_j,v_iv_j)\\
& = \Delta(v_i,v_i)\Delta(v_i,v_j)\Delta(v_j,v_i)\Delta(v_j,v_j)\\
&=\Delta(v_i,v_j)\Delta(v_j,v_i)\\
& = \Delta(v_i,v_j)^2,
\end{align*}
where the last equality holds because $\Delta$ is symmetric. Since $p$ is odd, we may take the square root and so $\Delta(v_i,v_j)=\Delta(v_j,v_i)=1$.
\end{proof}

\begin{prop}\label{S=1} We have $S=1$ in all cases (a),(b),(c),(d), and (e).
\end{prop}

\begin{proof}Let $\Delta\in S$ be arbitrary. We consider each case separately.
\begin{enumerate}[label=(\alph*),wide=0pt]
\item It is clear from Proposition \ref{auto1'} that
\[ v_1^{\alpha_{12}} = v_1v_2\mbox{ and } v_3^{\alpha_{23}} = v_2v_3\]
for some $\alpha_{12},\alpha_{23}\in P$. We then have
\[ \Delta(v_i,v_j) = 1\mbox{ for all }1\leq i \leq j\leq 3\mbox{ with }(i,j)\neq (1,3) \]
by Proposition \ref{prelim prop sym} and Lemma \ref{sym lemma}. Comparing the irreducible components of $S^2V$ and $\Lambda^2V$ as $Q$-modules, we also see that
\[ \Delta(v_1,v_3) = (v_1\wedge v_3)^\lambda\]
for some $\lambda\in\mathbb{F}_p$. But consider the action of $\alpha\in P$ given by
\[ \alpha = \begin{bmatrix} 1 & 1 & 0 \\ 0 & 1 & 0 \\ 0 & 1 & 1\end{bmatrix}.\]
By the condition (\ref{Delta3}), we have
\begin{align*}
\Delta(v_1,v_3)^{\hat{\alpha}}  & = \Delta(v_1^\alpha,v_3^\alpha)\\
& = \Delta(v_1v_2,v_2v_3)\\
& = \Delta(v_1,v_2)\Delta(v_1,v_3)\Delta(v_2,v_2)\Delta(v_2,v_3)\\
& = \Delta(v_1,v_3).
\end{align*}
But the left hand side is equal to
\[(v_1v_2\wedge v_2v_3)^\lambda =  (v_1\wedge v_2)^\lambda  (v_2\wedge v_3)^\lambda\Delta(v_1,v_3).\]
It follows that $\lambda=0$ and so $\Delta(v_1,v_3)=1$ also holds.
\item It is clear from Proposition \ref{auto2'} that
\[ v_1^{\alpha_{12}} = v_1v_2\mbox{ and } v_1^{\alpha_{13}} = v_1v_3\]
for some $\alpha_{12},\alpha_{13}\in P$. We then have
\[ \Delta(v_i,v_j) = 1\mbox{ for all }1\leq i \leq j\leq 3 \]
by Proposition \ref{prelim prop sym} and Lemma \ref{sym lemma}. 
\item It is clear from Proposition \ref{auto1} that
\[ v_1^{\alpha_{12}} = v_1v_2,\, 
v_3^{\alpha_{23}} = v_2v_3,\, v_4^{\alpha_{24}} = v_2v_4\]
for some $\alpha_{12},\alpha_{23},\alpha_{24}\in P$. We then have
 \[ \Delta(v_i,v_j)=1\mbox{ for all }1\leq i \leq j \leq 4 \mbox{ with }(i,j)\not\in\{(1,3),(1,4)\}\]
by Proposition \ref{prelim prop sym} and Lemma \ref{sym lemma}. Comparing the irreducible components of $S^2V$ and $\Lambda^2V$ as $Q$-modules, we also see that
\[ \Delta(v_1,v_3),\Delta(v_1,v_4)\in \langle v_1\wedge v_3,v_1\wedge v_4\rangle\]
has to hold. Let us write
\[ \Delta(v_1, v_3) = (v_1\wedge v_3)^{\lambda}(v_1\wedge v_4)^{\kappa},\]
and consider the action of $\alpha_1\in P$ defined by
\[\alpha_1 =  \begin{bmatrix}1 & 1 & 0 & 0 \\
0 & 1 & 0 & 0\\
 0& 1 & 1 & 0\\
 0 & 0 & 0 & 1\end{bmatrix}.
 \]
Since $\Delta$ satisfies the condition (\ref{Delta3}), we get that
\begin{align*}
\Delta(v_1,v_3)^{\hat{\alpha}_1}
& = \Delta(v_1^{\alpha_1},v_3^{\alpha_1})\\
& = \Delta(v_1v_2,v_2v_3)\\
& =\Delta(v_1,v_2)\Delta(v_1,v_3)\Delta(v_2,v_2)\Delta(v_2,v_3)\\
& = \Delta(v_1,v_3).\end{align*}
But explicitly, the left hand side is given by
\[(v_1v_2\wedge v_2v_3)^\lambda (v_1v_2\wedge v_4)^{\kappa}  = (v_1\wedge v_2)^{\lambda} (v_2\wedge v_3)^\lambda(v_2\wedge v_4)^\kappa\Delta(v_1,v_3).\]
This shows that $\lambda = \kappa = 0$ and hence $\Delta(v_1,v_3) =1$. Since there exists $\alpha_2\in Q$ for which $v_1^{\alpha_2} = v_1$ and $v_3^{\alpha_2} = v_4$, we have
\[ 1 = \Delta(v_1,v_3)^{\hat{\alpha}_2} = \Delta(v_1^{\alpha_2},v_3^{\alpha_2} ) = \Delta(v_1,v_4).\]
We have thus shown that $\Delta(v_1,v_3) = \Delta(v_1,v_4)=1$ also holds.
\item It is clear from Proposition \ref{auto2} that 
\[ v_1^{\alpha_{12}} = v_1v_2,\,
v_1^{\alpha_{13}} = v_1v_3,\,
v_1^{\alpha_{14}} = v_1v_4,\,
v_2^{\alpha_{23}} = v_2v_3,\,
v_2^{\alpha_{24}} = v_2v_4\]
for some $\alpha_{12},\alpha_{13},\alpha_{14},\alpha_{23},\alpha_{24}\in P$. We then have
\[ \Delta(v_i,v_j) = 1\mbox{ for all }1\leq i \leq j\leq 4 \]
by Proposition \ref{prelim prop sym} and Lemma \ref{sym lemma}.  
\item It is clear from Proposition \ref{auto3} that
\[ v_1^{\alpha_{12}} = v_1v_2,\,
v_1^{\alpha_{13}} = v_1v_3,\,
v_1^{\alpha_{14}} = v_1v_4,\,
v_3^{\alpha_{23}} = v_2v_3,\,
v_4^{\alpha_{24}}= v_2v_4\]
for some $\alpha_{12},\alpha_{13},\alpha_{14},\alpha_{23},\alpha_{24}\in P$. We then have
\[ \Delta(v_i,v_j) = 1\mbox{ for all }1\leq i \leq j\leq 4 \]
by Proposition \ref{prelim prop sym} and Lemma \ref{sym lemma}.  
\end{enumerate}
In all cases, we have shown that $\Delta=1$, and so indeed $S=1$.
  \end{proof}
 
Next, note that an anti-symmetric bilinear form $\Delta :V \times V \rightarrow \Lambda^2V$ is uniquely determined by
\[ \Delta(v_i,v_j) \mbox{ for }1\leq i < j\leq n.\]
We also make the following observation.

\begin{lemma}\label{anti lemma}
Let $\Delta\in S'$ and let $1\leq i,j,k\leq n$ with $i\neq j,k$. If
\begin{enumerate}[label = $(\arabic*)$]
\item $\Delta(v_i,v_j) = (v_i\wedge v_j)^{\lambda_1}$ or equivalently $\Delta(v_j,v_i) = (v_j\wedge v_i)^{\lambda_1}$,
\item $\Delta(v_i,v_k) = (v_i\wedge v_k)^{\lambda_2}$ or equivalently $\Delta(v_k,v_i) = (v_k\wedge v_i)^{\lambda_2}$,
\item $v_i^\alpha = v_i,\, v_j^\alpha =v_jv_k$ for some $\alpha\in \Aut^c(\pi)$,\end{enumerate}
then $\lambda_1 = \lambda_2$ has to hold.
\end{lemma}


\begin{proof}By the condition (\ref{Delta3}), we have
\[
\Delta(v_i,v_j)^{\hat{\alpha}} = \Delta(v_i^\alpha,v_j^\alpha)
= \Delta(v_i,v_jv_k)
=\Delta(v_i,v_j)\Delta(v_i,v_k). \]
Using the hypothesis, we rewrite this as
\[ (v_i\wedge v_j)^{\lambda_1}(v_i\wedge v_k)^{\lambda_1} = (v_i\wedge v_j)^{\lambda_1}( v_i\wedge v_k)^{\lambda_2},\] 
which implies that $\lambda_1 =\lambda_2$, as claimed.
\end{proof}

For each $\lambda\in\mathbb{F}_p$, as noted in Remark \ref{remark}, clearly
\[ \Delta_{[\lambda]} : V \times V\rightarrow \Lambda^2V;\,\ \Delta_{[\lambda]}(u,v) = (u\wedge v)^\lambda\]
is an anti-symmetric bilinear form satisfying (\ref{Delta3}), namely $\Delta_{[\lambda]}\in S'$.

\begin{prop}\label{S' prop}We have
\[ S' =\begin{cases}
  \{ \Delta_{[\lambda]} : \lambda\in \mathbb{F}_p\}&\mbox{in cases (a),(b),(c), and (d)},\\
  \{ \Delta_{[\lambda]}\Delta_{[\kappa]}^* : \lambda,\kappa\in \mathbb{F}_p\}&\mbox{in case (e)},
  \end{cases} \] 
where $\Delta_{[\kappa]}^* : V\times V\rightarrow \Lambda^2V$ denotes the anti-symmetric form defined by
\begin{align*}
\Delta_{[\kappa]}^*(v_1,v_2) & = (v_3\wedge v_4)^\kappa,&\Delta_{[\kappa]}^*(v_2,v_3) & = (v_2\wedge v_3)^{-\kappa},\\
\Delta_{[\kappa]}^*(v_1,v_3) & = (v_1\wedge v_3)^{-\kappa},&\Delta_{[\kappa]}^*(v_2,v_4)& = (v_2\wedge v_4)^{-\kappa},\\
\Delta_{[\kappa]}^*(v_1,v_4)& = (v_1\wedge v_4)^{-\kappa},&\Delta_{[\kappa]}^*(v_3,v_4) & = (v_1\wedge v_2)^{\kappa}.\end{align*}
\end{prop}

\begin{proof}Let $\Delta\in S'$ be arbitrary. We consider each case separately.
\begin{enumerate}[label=(\alph*),wide=0pt]
\item[(a),(b)] By Propositions \ref{prelim prop anti} and \ref{scalar prop}, we know that
\begin{align*} \Delta(v_1,v_2) &= (v_1\wedge v_2)^{\lambda_1}\\ 
\Delta(v_1,v_3) &= (v_1\wedge v_3)^{\lambda_2}\\
\Delta(v_2,v_3) &= (v_2\wedge v_3)^{\lambda_3}
\end{align*}
for some $\lambda_1,\lambda_2,\lambda_3 \in \mathbb{F}_p$. In case (a), by Proposition \ref{auto1'}, we have
\[ \begin{cases}
v_1^{\alpha_{12}} = v_1\\
v_3^{\alpha_{12}} = v_2v_3
\end{cases}\,\ \begin{cases}
v_3^{\alpha_{23}} = v_3\\
v_1^{\alpha_{23}} = v_1v_2
\end{cases}\]
for some $\alpha_{12},\alpha_{23}\in P$. In case (b), we already know from Proposition  \ref{scalar prop} that $\lambda_1=\lambda_2$, and by Proposition \ref{auto2'}, we have
\[ \begin{cases}
v_3^{\alpha_{23}} = v_3\\
v_1^{\alpha_{23}} = v_1v_2
\end{cases}\]
for some $\alpha_{23}\in P$. In both cases, we get that
\[\lambda :=\lambda_1 = \lambda_2 = \lambda_3\]
by Lemma \ref{anti lemma}. This shows that $\Delta = \Delta_{[\lambda]}$, as claimed.
\item[(c),(d)] By Propositions \ref{prelim prop anti} and \ref{scalar prop}, we know that
\begin{align*} \Delta(v_1,v_2) &= (v_1\wedge v_2)^{\lambda_1}&\Delta(v_2,v_3) & = (v_2\wedge v_3)^{\lambda_3}\\ 
\Delta(v_1,v_3) &= (v_1\wedge v_3)^{\lambda_2} & \Delta(v_2,v_4) &= (v_2\wedge v_4)^{\lambda_3}\\
\Delta(v_1,v_4) &= (v_1\wedge v_4)^{\lambda_2}&\Delta(v_3,v_4) &= (v_3\wedge v_4)^{\lambda_4}
\end{align*}
for some $\lambda_1,\lambda_2,\lambda_3,\lambda_4 \in \mathbb{F}_p$. In case (c), by Proposition \ref{auto1}, we have
\[ \begin{cases}
v_1^{\alpha_{12}} = v_1\\
v_3^{\alpha_{12}} = v_2v_3
\end{cases}\,\
\begin{cases}
v_3^{\alpha_{23}} = v_3\\
v_1^{\alpha_{23}} = v_1v_2
\end{cases}
\,\
\begin{cases}
v_4^{\alpha_{34}} = v_4\\
v_3^{\alpha_{34}} = v_2v_3
\end{cases}\]
for some $\alpha_{12},\alpha_{23},\alpha_{34}\in P$. In case (d), by Proposition \ref{auto2}, we have
\[ \begin{cases}
v_1^{\alpha_{12}} = v_1\\
v_2^{\alpha_{12}} = v_2v_3
\end{cases}\,\
\begin{cases}
v_3^{\alpha_{23}} = v_3\\
v_1^{\alpha_{23}} = v_1v_2
\end{cases}
\,\
\begin{cases}
v_4^{\alpha_{34}} = v_4\\
v_2^{\alpha_{34}} = v_2v_3
\end{cases}\]
for some $\alpha_{12},\alpha_{23},\alpha_{34}\in P$. In both cases, we get that 
\[\lambda :=\lambda_1 = \lambda_2 = \lambda_3= \lambda_4\]
by Lemma \ref{anti lemma}. This shows that $\Delta = \Delta_{[\lambda]}$, as claimed.
\item[(e)] By Propositions \ref{prelim prop anti} and \ref{scalar prop}, we know that 
\begin{align*} \Delta(v_1,v_2) &= (v_1\wedge v_2)^{\lambda_1}(v_3\wedge v_4)^{\kappa_1}&\Delta(v_2,v_3) & = (v_2\wedge v_3)^{\lambda_3}\\ 
\Delta(v_1,v_3) &= (v_1\wedge v_3)^{\lambda_2} & \Delta(v_2,v_4) &= (v_2\wedge v_4)^{\lambda_3}\\
\Delta(v_1,v_4) &= (v_1\wedge v_4)^{\lambda_2}&\Delta(v_3,v_4) &= (v_1\wedge v_2)^{\kappa_4}(v_3\wedge v_4)^{\lambda_4}
\end{align*}
 for some $\lambda_1,\lambda_2,\lambda_3,\lambda_4,\kappa_1,\kappa_4\in \mathbb{F}_p$. Consider $\alpha\in P$ given by
\[ \alpha = \begin{bmatrix} 1 & 0 & 0 & 1\\
0 & 1 & 0 & 0\\
 0 & 1 & 1 & 0 \\ 
 0 & 0 & 0 & 1\end{bmatrix},\]
and we compute that
 \begin{align*}
\Delta(v_1,v_2)^{\hat{\alpha}}& = (v_1v_4\wedge v_2)^{\lambda_1}(v_2v_3\wedge v_4)^{\kappa_1} \\
&= \Delta(v_1,v_2)(v_4\wedge v_2)^{\lambda_1-\kappa_1},\\
\Delta(v_1^\alpha,v_2^\alpha) & = \Delta(v_1v_4,v_2) \\
&=\Delta(v_1,v_2)(v_4\wedge v_2)^{\lambda_3},\\
\Delta(v_1,v_3)^{\hat{\alpha}} & = (v_1v_4\wedge v_2v_3)^{\lambda_2} \\
&= \Delta(v_1,v_3)(v_1\wedge v_2)^{\lambda_2}(v_4\wedge v_2)^{\lambda_2}(v_4\wedge v_3)^{\lambda_2},\\
\Delta(v_1^\alpha,v_3^\alpha) & = \Delta(v_1v_4,v_2v_3) \\
&= \Delta(v_1,v_3)(v_1\wedge v_2)^{\lambda_1-\kappa_4}(v_4\wedge v_2)^{\lambda_3}(v_4\wedge v_3)^{\lambda_4-\kappa_1},\\
\Delta(v_3,v_4)^{\hat{\alpha}} & = (v_1v_4\wedge v_2)^{\kappa_4}(v_2v_3\wedge v_4)^{\lambda_4}\\
& = \Delta(v_3,v_4)(v_2\wedge v_4)^{\lambda_4-\kappa_4},\\
\Delta(v_3^\alpha,v_4^\alpha) & = \Delta(v_2v_3,v_4)\\
&= \Delta(v_3,v_4)(v_2\wedge v_4)^{\lambda_3}.
\end{align*}
Since the condition (\ref{Delta3}) has to hold, we deduce that
\[ \lambda_3= \lambda_1 - \kappa_1,\,\
\lambda_2 = \lambda_1-\kappa_4 =\lambda_3=\lambda_4-\kappa_1,\,\ \lambda_3 = \lambda_4-\kappa_4.\]
Solving this system of equations, we get that
\[ \lambda := \lambda_1 = \lambda_4 ,\,\ \kappa:=\kappa_1=\kappa_4,\,\ \lambda_2 =\lambda_3 = \lambda -\kappa.\]
This shows that $\Delta = \Delta_{[\lambda]}\Delta_{[\kappa]}^*$. Conversely, for any $\lambda,\kappa\in\mathbb{F}_p$, we know that $ \Delta_{[\lambda]}\in S'$ already and it is straightforward to check that $\Delta_{[\kappa]}^*$ also satisfies (\ref{Delta3}), so then $\Delta_{[\lambda]}\Delta_{[\kappa]}^*\in S'$.
 \end{enumerate}
 This completes the proof.
\end{proof}
   
\subsection{The structure of $T(G)$} We shall now prove Theorem \ref{thm1}. We already know from (\ref{T(G)}) and Proposition \ref{S=1} that
\[ T(G) \simeq \res(\mathcal{S}').\]
In cases (a),(b),(c), and (d), the theorem follows because we have
\[ \res(\mathcal{S}') \simeq \mathbb{F}_p^\times\]
by Remark \ref{remark} and Proposition \ref{S' prop}. In case (e), by Proposition \ref{S' prop}, the elements of $S'$  are precisely the bilinear forms
\[ \Delta_{[\sigma]}: V\times V\rightarrow\Lambda^2V ;\,\  \Delta_{[\sigma]}(u,v) = (u\wedge v)^\sigma.\]
Here $\sigma$ is any endomorphism on $\Lambda^2V$ of the form
\begin{equation}\label{tau}
 \begin{bmatrix}
\lambda & &  & &&\kappa\\
 & \lambda-\kappa & & &&\\
 & & \lambda-\kappa & & &\\
 & & & \lambda-\kappa & &\\
 & & & &\lambda-\kappa &\\
\kappa & & & &&\lambda
\end{bmatrix} \mbox{ with }\lambda,\kappa\in \mathbb{F}_p,\end{equation}
written with respect to the basis
\[ v_1\wedge v_2, v_1\wedge v_3, v_1\wedge v_4,v_2\wedge v_3,v_2\wedge v_4,v_3\wedge v_4\]
of $\Lambda^2V$. By \cite[Example 3.4]{LMH}, we know that $N_{\Delta_{[\sigma]}}\simeq G$ occurs only for $1+2\sigma\in \GL(\Lambda^2V)$. Let us make a change of variables $\tau = 1+2\sigma$, and consider $\tau_{\lambda,\kappa}\in \GL(\Lambda^2V)$ of the form (\ref{tau}) but with the restriction $\kappa\neq\pm\lambda$. Observe that then
\[
\eta_{\lambda,\kappa}   = \begin{bmatrix}
\lambda+\kappa &&&\\
&(\lambda+\kappa)^{-1} &&\\
&&1&\\
&&&1
\end{bmatrix},\]
written with respect to the basis $v_1,v_2,v_3,v_4$ of $V$, in which case
\[\hat{\eta}_{\lambda,\kappa} = \begin{bmatrix}
1 &&&&&\\
 &\lambda+\kappa&&&&\\
 &&\lambda+\kappa&&&\\
 &&&(\lambda+\kappa)^{-1}&&\\
 &&&&(\lambda+\kappa)^{-1}&\\
 &&&&&1
\end{bmatrix},\]
yields a solution to $\pi\hat{\eta}_{\lambda,\kappa}\tau_{\lambda,\kappa} = \eta_{\lambda,\kappa}\pi$. From (\ref{S'}), we deduce that
\[ \res(\mathcal{S}') \simeq \{(\eta_{\lambda,\kappa},\hat{\eta}_{\lambda,\kappa}\tau_{\lambda,\kappa}) : \lambda,\kappa\in\mathbb{F}_p\mbox{ with }\kappa\neq\pm\lambda\}.  \]
It is straightforward to verify that
\[ \eta_{\lambda_{1},\kappa_1} \eta_{\lambda_{2},\kappa_{2}} = \eta_{\lambda,\kappa},\,\
 \hat{\eta}_{\lambda_1,\kappa_1}\tau_{\lambda_1,\kappa_1}\hat{\eta}_{\lambda_2,\kappa_2}\tau_{\lambda_2,\kappa_2} = \hat{\eta}_{\lambda,\kappa}\tau_{\lambda,\kappa}\]
 for any $\lambda_1,\lambda_2,\kappa_1,\kappa_2\in \mathbb{F}_p$ with $\kappa_1\neq\pm\lambda_1$ and $\kappa_2\neq\pm\lambda_2$, where
 \[\begin{bmatrix} \lambda & \kappa\\ \kappa &\lambda \end{bmatrix}= \begin{bmatrix} \lambda_1 &\kappa_1\\\kappa_1&\lambda_1\end{bmatrix}
 \begin{bmatrix}\lambda_2& \kappa_2\\ \kappa_2 & \lambda_2\end{bmatrix}
 =\begin{bmatrix} \lambda_1\lambda_2 + \kappa_1\kappa_2 & \lambda_1\kappa_2 +\lambda_2\kappa_1\\\lambda_1\kappa_2 +\lambda_2\kappa_1&\lambda_1\lambda_2 + \kappa_1\kappa_2 \end{bmatrix}.\]
It follows that $\res(\mathcal{S}')$ is isomorphic to the subgroup
\[ \left\{ \begin{bmatrix}\lambda & \kappa \\ \kappa & \lambda\end{bmatrix}  : \lambda,\kappa\in\mathbb{F}_p\mbox{ with }\kappa\neq\pm\lambda\right\}\]
of $\GL_2(\mathbb{F}_p)$, or conjugating by $\left[\begin{smallmatrix}1 & -1\\ 1 & 1 \end{smallmatrix}\right]$, the subgroup
\[ \left\{ \begin{bmatrix}\lambda + \kappa& 0 \\  0 & \lambda- \kappa\end{bmatrix}  : \lambda,\kappa\in\mathbb{F}_p\mbox{ with }\kappa\neq\pm\lambda\right\}\]
of $\GL_2(\mathbb{F}_p)$. This decomposes as
\[ \left\{ \begin{bmatrix} \lambda & 0 \\ 0 & 1 \end{bmatrix}: \lambda\in \mathbb{F}_p^\times \right\}\times  \left\{ \begin{bmatrix} 1 & 0 \\ 0 & \kappa \end{bmatrix}: \kappa\in \mathbb{F}_p^\times \right\}\]
and so is isomorphic to $\mathbb{F}_p^\times \times \mathbb{F}_p^\times$, as claimed in (e).

\bibliographystyle{amsalpha}
 
\bibliography{Refs}

\providecommand{\bysame}{\leavevmode\hbox to3em{\hrulefill}\thinspace}
\providecommand{\MR}{\relax\ifhmode\unskip\space\fi MR }
\providecommand{\MRhref}[2]{%
  \href{http://www.ams.org/mathscinet-getitem?mr=#1}{#2}
}
\providecommand{\href}[2]{#2}
\begin{thebibliography}{CDV18}

\bibitem[Car15]{CarIJM}
A.~Caranti, \emph{A module-theoretic approach to abelian automorphism groups},
  Israel J. Math. \textbf{205} (2015), no.~1, 235--246. \MR{3314589}

\bibitem[Car16]{simple-constr}
\bysame, \emph{A simple construction for a class of {$p$}-groups with all of
  their automorphisms central}, Rend. Semin. Mat. Univ. Padova \textbf{135}
  (2016), 251--258. \MR{3506071}

\bibitem[Car18]{class2}
\bysame, \emph{Multiple holomorphs of finite {$p$}-groups of class two}, J.
  Algebra \textbf{516} (2018), 352--372. \MR{3863482}

\bibitem[CDV17]{fg}
A.~Caranti and F.~Dalla~Volta, \emph{The multiple holomorph of a finitely
  generated abelian group}, J. Algebra \textbf{481} (2017), 327--347.
  \MR{3639478}

\bibitem[CDV18]{perfect}
\bysame, \emph{Groups that have the same holomorph as a finite perfect group},
  J. Algebra \textbf{507} (2018), 81--102. \MR{3807043}

\bibitem[CT23]{LMH}
A.~Caranti and C.~Tsang, \emph{Finite $p$-groups of class two with a large
  multiple holomorph}, J. Algebra \textbf{617} (2023), 476--499.

\bibitem[Koh15]{Kohl}
Timothy Kohl, \emph{Multiple holomorphs of dihedral and quaternionic groups},
  Comm. Algebra \textbf{43} (2015), no.~10, 4290--4304. \MR{3366576}

\bibitem[Mil08]{Miller}
G.~A. Miller, \emph{On the multiple holomorphs of a group}, Math. Ann.
  \textbf{66} (1908), no.~1, 133--142. \MR{1511494}

\bibitem[Mil51]{Mills}
W.~H. Mills, \emph{Multiple holomorphs of finitely generated abelian groups},
  Trans. Amer. Math. Soc. \textbf{71} (1951), 379--392. \MR{45117}

\bibitem[Tsa20]{squarefree}
Cindy Tsang, \emph{On the multiple holomorph of groups of squarefree or odd
  prime power order}, J. Algebra \textbf{544} (2020), 1--28. \MR{4023139}

\end{thebibliography}

\end{document}